\documentclass[11pt]{amsart}
\usepackage[mathscr]{eucal}
\usepackage{palatino, mathpazo, amsfonts, mathrsfs, amscd}
\usepackage[all]{xy}
\usepackage{url}
\usepackage{amssymb, amsmath, amsthm}
\usepackage{stackrel}
\usepackage{bm}

\usepackage{tikz-cd}
\usepackage{tikz}
\usetikzlibrary{arrows}

\newtheorem{theo}{Theorem}[section]
\newtheorem{lemma}[theo]{Lemma}
\newtheorem{prop}[theo]{Proposition}
\newtheorem{coro}[theo]{Corollary}
\theoremstyle{remark}
\newtheorem{rema}[theo]{Remark}

\newtheorem{defi}[theo]{Definition}
\newtheorem{nota}[theo]{Notation}
\newtheorem{assumption}[theo]{Assumption}

\numberwithin{equation}{subsection}
\usepackage{todonotes}

\makeatletter
\def\imod#1{\allowbreak\mkern10mu({\operator@font mod}\,\,#1)}
\makeatother

\newcommand{\cc}[1]{\mathcal{#1}}  

\newcommand{\CC}{\mathbb{C}}
\newcommand{\ZZ}{\mathbb{Z}}

\newcommand{\PP}{\mathbb{P}}
\newcommand{\QQ}{\mathbb{Q}}

\newcommand{\RR}{\mathbb{R}}

\newcommand{\sMbar}{\overline{\mathscr{M}}}

\newcommand{\bq}{\bm{q}}

 \newcommand{\bp}{\mathbf{p}}
  
\newcommand{\bt}{{\bm{t}}}

\newcommand{\bs}{\bm{s}}

\newcommand{\cX}{\cc X}

\newcommand{\cY}{\cc Y}
\newcommand{\sV}{{\mathscr V}}
\newcommand{\sL}{{\mathscr L}}
\newcommand{\cC}{\cc C}
\newcommand{\cE}{\cc E}

\newcommand{\cZ}{\cc Z}

\DeclareMathOperator{\ch}{\widetilde{ch}}
\DeclareMathOperator{\chr}{ch}

\DeclareMathOperator{\tot}{tot}

\newcommand{\br}[1]{\left\langle#1\right\rangle}  

\newcommand{\set}[1]{\left\{#1\right\}}  
\newcommand{\op}[1]{\operatorname{#1}}

\title{Narrow quantum $D$-modules and quantum Serre duality}
\author{{Mark} {Shoemaker}}
\address{
   Colorado State University \\
   Department of Mathematics \\
   1874 Campus Delivery \\
    Fort Collins, CO, USA, 80523-1874
}
\email{mark.shoemaker@colostate.edu}

\begin{document}

\begin{abstract}
Given $\cY$ a non-compact manifold or orbifold, we define a natural subspace of the cohomology of $\cY$ called the narrow cohomology.  We show that despite $\cY$ being non-compact, there is a well-defined and non-degenerate pairing on this subspace.  The narrow cohomology proves useful for the study of genus zero Gromov--Witten theory.  When $\cY$ is a smooth complex variety or Deligne--Mumford stack, one can define a quantum $D$-module on the narrow cohomology of $\cY$.  This yields a new formulation of quantum Serre duality.
\end{abstract}

\maketitle
\tableofcontents

\section{Introduction}
Let $\cX$ be a smooth complex variety (or orbifold) and let $\cE \to \cX$ be a vector bundle over $\cX$, with a regular section $s \in \Gamma( \cX, \cE)$.  We can define the subvarieties
\[ \cZ := \{s = 0\} \subset \cX\hspace{.3 cm} \text{  and  } \hspace{.3 cm}\cY := \tot(\cE^\vee),\] to obtain
 the following diagram
\begin{equation}\label{e1}\begin{tikzcd}
& \ar[d, "\pi"] \cY  \\
\cZ \ar[r, hookrightarrow, "j"] &\cX.
\end{tikzcd}\end{equation}
As was originally observed in \cite{G1} in the case where $\cX$ is projective space $\PP^N$, there is a nontrivial relationship between the genus zero Gromov--Witten invariants of $\cZ$ and those of $\cY$.  This correspondence was given the name \emph{quantum Serre duality}.  It was originally proven by showing that (equivariant lifts of) generating functions of Gromov--Witten invariants of $\cZ$ and $\cY$ obey very similar recursions.  
\subsection{Other incarnations}
The correspondence has since been generalized and rephrased many times.  In \cite{CG}, it was generalized to the case of an arbitrary base $\cX$ using 
Givental's Lagrangian cones.  In this formulation, one endows $\cY$ with a $\CC^*$-action by scaling fibers.  Then
 the $\CC^*$-equivariant genus zero Gromov--Witten theory of $\cY$ can be used to recover the genus zero Gromov--Witten theory of $\cZ$.  This result forms the basis (and the proof) of later formulations.

%
%
%

Recently in \cite{IMM}, quantum Serre duality was re-expressed as a correspondence between \emph{quantum $D$-modules}.   The quantum $D$-module $QDM(\cX)$ for a (compact) space $\cX$ consists of 
\begin{itemize}
\item the quantum connection:\footnote{In the case where $\cZ$ and $\cY$ are orbifolds we should replace cohomology with \emph{Chen--Ruan} cohomology.  See \S~\ref{ss:orb} for details.}
 \[\nabla^\cX: H^*(\cX) \to H^*(\cX) [\bt, z, z^{-1}],\] a flat connection defined in terms of the quantum product $\bullet_\bt$;
\item a pairing: \[S^\cX(-,-): H^*(\cX)[\bt, z, z^{-1}] \otimes H^*(\cX)[\bt, z, z^{-1}] \to \CC[z, z^{-1}],\]
which is flat with respect to the connection.
\end{itemize}  See Definition~\ref{d:qdm} for details.  
As with Givental's Lagrangian cone, the quantum $D$-module of $\cX$ fully determines the genus zero Gromov--Witten theory of $\cX$.
In \cite{Iri}, Iritani defined a corresponding \emph{integral structure}, a lattice lying inside the kernel of $\nabla^\cX$ defined as the image of the bounded derived category $D(\cX)$ under a map 
\[\bs^\cX: D(\cX) \to H^*(\cX)[\bt, z, z^{-1}].
\] See \S~\ref{ss:fs} for details.

One benefit to the formulation in terms of quantum $D$-modules \cite{IMM}, is that quantum Serre duality can be phrased non-equivariantly, as the quantum connection $\nabla^\cY$ is well-defined without the need for  $\CC^*$-localization.
It is shown (Theorem 3.14, Corollary 3.17) that the map
\begin{equation}\label{e2} \pi^* \circ j_*: H^*_{\op{amb}}(\cZ)[\bt, z, z^{-1}] \to H^*(\cY)[\bt, z, z^{-1}]\end{equation}
sends
$\ker\left(\nabla^\cZ \right)$ to $\ker\left(\nabla^\cY \right)$ after a change of variables.  Here $H^*_{\op{amb}}(\cZ)$ denotes the image $j^*(H^*(\cX)) \subset H^*(\cZ)$.  Furthermore, they prove that this map is compatible with the integral structures and with the functor \[(-1)^{\op{rk}}(\cE)\op{det}(\cE)\otimes (\pi^* \circ j_*)(-): D(\cZ) \to D(\cY).\]

It is important to note, however, that \eqref{e2} does not give a map of quantum $D$-modules, because the pairing $S^\cY(-,-)$ is not well-defined when $\cY$ is non-compact.  
Nevertheless, the formulation of quantum Serre duality in terms of the quantum connection seems to be a natural and useful way of viewing the correspondence.

\subsection{Results}
The work described above raises the following interrelated questions:
\begin{enumerate}
\item Is there a well-defined quantum $D$-module associated to $\cY$ when $\cY$ is non-compact?  More precisely,  can one define a pairing which is flat with respect to the quantum connection?
\item If (1) holds, can quantum Serre duality be rephrased as a map between quantum $D$-modules, identifying not just the quantum connection but also the pairings?
\item The map in \eqref{e2} is not an isomorphism; under mild assumptions it is an inclusion.  Is there a natural way of describing the image $\pi^* \circ j_*\left(\ker\left(\nabla^\cZ \right)\right)$ inside $\ker\left(\nabla^\cY \right)$?
\end{enumerate}

This paper answers each of these questions in the affirmative.  Given $\cY$ a non-compact smooth variety, we define a subspace $H^*_{\op{nar}}(\cY) \subset H^*(\cY)$ which we call the \emph{narrow cohomology} of $\cY$.  There is a natural forgetful morphism 
\[\phi: H^*_{\op{c}}(\cY) \to H^*(\cY)\]
from compactly supported cohomology to cohomology.  The narrow cohomology $H^*_{\op{nar}}(\cY)$ is defined to be the image of $\phi$.  This turns out to be the right framework in which to formulate quantum Serre duality as an \emph{isomorphism} of quantum $D$-modules.

One observes that the Poincar\'e pairing between $H^*_{\op{c}}(\cY)$ and $H^*(\cY)$ induces a non-degenerate pairing on $H^*_{\op{nar}}(\cY)$.  We use this to define a quantum $D$-module on the narrow subspace.
\begin{theo}[Corollary~\ref{c:qsdnar}]
The quantum connection \[\nabla^\cY: H^*(\cY) \to H^*(\cY) [[\bt]][ z, z^{-1}]\] preserves $H^*_{\op{nar}}(\cY)$.  Furthermore, there is a well-defined and nondegenerate pairing 
\[S^{\cY, \op{nar}}(-,-): H^*_{\op{nar}}(\cY) [[\bt]][ z, z^{-1}] \otimes H^*_{\op{nar}}(\cY) [[\bt]][ z, z^{-1}] \to \CC[z, z^{-1}],\]
which is flat with respect to $\nabla^\cY$.  We obtain a quantum $D$-module $QDM_{\op{nar}}(\cY)$ on $H^*_{\op{nar}}(\cY)$.
\end{theo}
There is also a well-defined integral structure for the narrow quantum $D$-module, coming from the derived category $D(\cY)_{\op{c}}$ of complexes of coherent sheaves on $\cY$, exact outside a proper subvariety.
The narrow quantum $D$-module defined above turns out to be particularly relevant to quantum Serre duality.  In the particular case of $\cY$ and $\cZ$ from \eqref{e1}, we show there is an isomorphism of quantum $D$-modules from $QDM_{\op{nar}}(\cY)$ to $QDM_{\op{amb}}(\cZ)$.  This is the main result of the paper.
\begin{theo}[Theorem~\ref{t:QSD}] \label{t01} There is an isomorphism \[\bar \Delta_+: H^*_{\op{nar}}(\cY)[z, z^{-1}] \to H^*_{\op{amb}}(\cZ)[z, z^{-1}]\]
which identifies the quantum $D$-module $QDM_{\op{nar}}(\cY)$ with $\bar f^* \left( QDM_{\op{amb}}(\cZ) \right)$ (where $\bar f$ is an explicit change of variables).  Furthermore it is compatible with the integral structure and the functor $j^* \circ \pi_*$, i.e., the following diagram commutes;
\begin{equation}
\begin{tikzcd}
D(\cY)_{\cX} \ar[r, " j^* \circ \pi_*"] \ar[d, "s^{\cY, \op{nar}}"] &  j^*(D(\cX))  \ar[d, "s^{\cZ, \op{amb}}"]\\
QDM_{\op{nar}}(\cY) \ar[r, "\bar \Delta_+"] & \bar f^* \left( QDM_{\op{amb}}(\cZ) \right).
\end{tikzcd}
\end{equation}
%
\end{theo}

In the course of proving the above theorem we make several interesting connections between the (non-equivariant) quantum connection on $\cY$ and different modifications of the Gromov--Witten theory over $\cX$.  

First, in \S~\ref{ss:ts} we give an explicit description of a modified quantum product, denoted $\bullet^{\cY \to \cX}_{\bt}$, on $\cX$ which pulls back to the usual (non-equivariant) quantum product $\bullet^\cY_{\bt}$ via $\pi^*$.  This result is dual in spirit to \cite[\S~2]{Pan} where the ``$\bullet_{\cZ}$ product induced by a hypersurface,'' a modified quantum product on $\cX$, is defined and related to the usual quantum product on $\cZ$.

Second, we consider the quantum connection with compact support on $\cY$:
\[ \nabla^{\cY, \op{c}}: H^*_{\op{c}}(\cY) \to H^*_{\op{c}}(\cY)[\bt, z, z^{-1}]\]
which is defined identically to $\nabla^\cY$ but acts on cohomology with compact support.  In a short remark \cite[Remark 3.17]{IMM}, it was reasoned by an analogy that (the non-equivariant limit of) the Euler-twisted quantum connection $\nabla^{e(\cE)}$ on  $\cX$ (see Definition~\ref{d:dub1}) could be thought of as the quantum connection with compact support.  In Proposition~\ref{p:Lcomc} we make this observation precise, showing that the pushforward isomorphism
\[ i^{\op{c}}_*: H^*_{\op{CR}}(\cX) \to H^*_{\op{CR}, \op{c}}(\cY)\]
identifies $\nabla^{e(\cE)}$ with $\nabla^{\cY, \op{c}}$ up to a change of variables.

The above results imply yet another variation of quantum Serre duality, which states, roughly, that 
\[j^* \circ \pi^{\op{c}}_*: H^*_{\op{CR}, \op{c}}(\cY) \to H^*_{\op{CR}, \op{amb}}(\cZ)\]
maps  $\ker\left(\nabla^{\cY, \op{c}}\right)$ to $\ker \left(\nabla^\cZ\right)$ and is compatible with integral structures and with the functor 
\[j^* \circ \pi_*:D(\cY)_{\cX} \to  D(\cZ).\]
See Theorem~\ref{t:dcom} for the precise statement.  This result is essentially the adjoint of the statement given in Theorem 3.13 of \cite{IMM}, once one observes $\nabla^{\cY, \op{c}}$ and $\nabla^\cY$ as dual with respect to the pairing between cohomology and compactly supported cohomology on $\cY$.
However there is a difference in the change of variables used in \cite{IMM} versus the theorem above.  
We give a direct proof of Theorem~\ref{t:dcom} in a slightly more general context than that appearing in \cite{IMM}, however the proof techniques are similar.  
See Remark~\ref{r:IMMr} for more details on the connection.  
\subsection{Applications}
We expect the constructions and results of this paper to be useful in formulating and proving new correspondences in genus zero Gromov--Witten theory.  For instance, with the narrow quantum $D$-module on hand, one can formulate a crepant transformation conjecture between non-compact $K$-equivalent spaces $\cY$ and $\cY '$.  Previously, most results along these lines have assumed that $\cY$ and $\cY '$ be toric varieties and have used equivariant Gromov--Witten theory for both the statement and proof of the correspondence.  In \cite{Sh} we show how, in a particular case, the equivariant correspondence implies the narrow correspondence.  This is generalized to toric wall crossings in \cite{RoSho}.

The formulation of quantum Serre duality as in Theorem~\ref{t01} will be useful as a tool for proving other correspondences.  In \cite{LPS}, we observed that the LG/CY correspondence is implied by a suitable version of the crepant transformation conjecture.  We use the shorthand ``CTC implies LG/CY.''  However the implication was somewhat messy to state in \cite{LPS}, as it required a careful analysis of the non-equivariant limits of certain maps and generating functions.  Furthermore it did not involve integral structures. 

In \cite{Sh}, we show that Theorem~\ref{t01} together with a sister result involving FJRW theory may be used in tandem to clarify the ``CTC implies LG/CY'' statement of \cite{LPS}.
We 
show that 
the narrow crepant transformation conjecture immediately implies a $D$-module formulation of the LG/CY correspondence (first described in \cite{CIR}).  This result requires no mention of equivariant Gromov--Witten theory, and is compatible with all integral structures.  In fact this was the first motivation for the current paper.

Another more recent application appears in  \cite{RoSho}, where R. Mi and the author us the results of this paper prove a comparison result for the Gromov--Witten theory of extremal transitions.

\subsection{Acknowledgments}  I am grateful to R.~Cavalieri, E.~Clader, J.~Gu\'er\'e, H.~Iritani, Y. P.~Lee, N.~Priddis, D.~Ross, Y.~Ruan and Y.~Shen for many useful conversations about Givental's symplectic formalism, quantum Serre duality, and quantum $D$-modules.   I also thank the anonymous referee for many helpful suggestions and comments.
This work was partially supported by NSF grant DMS-1708104.

\section{Narrow cohomology}\label{s:1}

Let $Y$ be a non-compact oriented manifold.  
Let $H^*(Y; \RR)$ and $H^*_{\op{c}}(Y; \RR)$ denote the de Rham cohomology and cohomology with compact support respectively.  We assume always that $Y$ has a finite good cover \cite[Section~5]{BT}, so that $\op{rank}(H_{\op{c}}^*(Y; \RR))$ and $\op{rank}(H^*(Y; \RR))$ are finite.
In this paper we will primarily be concerned with the case where $Y$ is the total space of a vector bundle on a compact manifold $X$.  In this case the existence of a finite good cover is automatic.

Let $\Omega^k(Y; \RR)$ and $\Omega^k_{\op{c}}(Y; \RR)$ denote the real vector space of $k$-forms and $k$-forms with compact support. 
We have the following commutative diagram:
\[
\begin{tikzcd}
0 \ar[r]  & \Omega^0_{\op{c}}(Y) \ar[r, "d"] \ar[d] & \Omega^1_{\op{c}}(Y) \ar[r, "d"] \ar[d]&  \Omega^2_{\op{c}}(Y) \ar[r, "d"] \ar[d] &\cdots \\
0 \ar[r] & \Omega^0(Y) \ar[r, "d"] & \Omega^1(Y) \ar[r, "d"] & \Omega^2(Y) \ar[r, "d"] & \cdots
\end{tikzcd} \]
where the vertical arrows are just the inclusion obtained by forgetting that a $k$-form had compact support.  This induce  ``forgetful'' maps $\phi_k: H^k_{\op{c}}(Y; \RR) \to H^k(Y; \RR)$ for $1 \leq k \leq \dim(Y)$.  Let $\phi: H^*_{\op{c}}(Y; \RR) \to H^*(Y; \RR)$ denote the direct sum of these maps.
\begin{defi}
We define the \emph{narrow} cohomology of $Y$ to be the image of $\phi$:
\[ H^k_{\op{nar}}(Y; \RR) := \op{im} \left( \phi_k: H^k_{\op{c}}(Y; \RR) \to H^k(Y; \RR) \right),\]
\[H^*_{\op{nar}}(Y; \RR) := \bigoplus_{k = 0}^{\dim(Y)} H^k_{\op{nar}}(Y; \RR) \subseteq H^*(Y; \RR).\]
\end{defi}
The subspace $\ker(\phi)$ consists of classes which can be represented by differential forms with compact support but which are boundaries of classes with non-compact support.
\begin{defi}
Given a class $\alpha \in H^k_{\op{nar}}(Y; \RR)$, we define a \emph{lift} of $\alpha$ to be any class $\tilde \alpha \in H^k_{\op{c}}(Y; \RR)$ such that $\phi_k(\tilde \alpha) = \alpha$.  Lifts are well defined up to a choice of degree $k$ element in $\ker(\phi)$.
\end{defi}

Using singular cohomology over an arbitrary coefficient ring $R$,
a completely analogous definition can be made via the inclusion 
\[ C^*_{\op{c}}(Y;R) \hookrightarrow C^*(Y; R).
\]
of the subchain complex $C^*_{\op{c}}(Y;R)$ of singular cochains with compact support into the chain complex $C^*(Y; R)$ of singular cochains with coefficients in $R$.  Via the induced map 
$\phi: H^*_{\op{c}}(Y; R) \to H^*(Y; R)$, one defines the narrow \emph{singular} cohomology
\[H^*_{\op{nar}}(Y; R) := \op{im}\left( \phi: H^*_{\op{c}}(Y; R) \to H^*(Y; R)\right).\]

It will be convenient to also introduce narrow homology.  This can be defined in terms of Borel--Moore homology (See e.g. Section~V of \cite{Br} for an exposition consistent with that described below).  Assume for simplicity that $Y$ is $\sigma$-compact to simplify our exposition.  Let $R$ be a ring and let $C_k(Y; R)$ denote the set of finite singular $k$-chains, consisting of \emph{finite} $R$-linear combinations 
\[\sum a_\sigma \sigma\]
of continuous maps $\sigma: \Delta^k \to Y$ with $a_\sigma \in R$.  
We define $C_k^{\op{BM}}(Y; R)$ to be the set of \emph{locally finite} singular $k$-chains, consisting of (possibly infinite)
$R$-linear combinations 
\[\sum a_\sigma \sigma\]
such that for each compact set $C \subset Y$, $a_\sigma$ is zero for all but finitely many of the maps $\sigma$ whose image meets $C$.  
In reasonable situations \cite{Br}, Borel--Moore homology can be defined as the homology of the complex $C_*^{\op{BM}}(Y; R)$.
Note we have a similar map of complexes as before
\[
\begin{tikzcd}
\cdots  \ar[r, "\partial"]& C_{i+1}(Y; R) \ar[r, "\partial"] \ar[d] & C_i(Y; R) \ar[r, "\partial"] \ar[d]&  C_{i-1}(Y; R) \ar[r, "\partial"] \ar[d] &\cdots \\
\cdots  \ar[r, "\partial"] &C_{i+1}^{\op{BM}}(Y; R) \ar[r, "\partial"] & C_i^{\op{BM}}(Y; R) \ar[r, "\partial"] & C_{i-1}^{\op{BM}}(Y; R) \ar[r, "\partial"] & \cdots
\end{tikzcd} \]
again inducing a ``forgetful'' map $\underline \phi: H_*(Y; R) \to H_*^{\op{BM}}(Y; R)$.
\begin{defi}
Define the \emph{narrow homology} to be the image of $\underline \phi$:
\[H_*^{\op{nar}}(Y; R):= \op{im}\left( \underline \phi: H_*(Y; R) \to H_*^{\op{BM}}(Y; R) \right).\]
\end{defi}


Poincar\'e duality \cite{Br} gives isomorphisms  
\begin{align*}
H^k_{\op{c}}(Y; \ZZ) &\cong H_{n-k}(Y; \ZZ) \\
H^k(Y; \ZZ) &\cong H_{n-k}^{\op{BM}}(Y; \ZZ).
\end{align*}
The following diagram commutes, 
\begin{equation}\label{e:PD}
\begin{tikzcd}
H^k_{\op{c}}(Y; \ZZ) \ar[r, "\op{PD}"] 
\ar[d, "\phi"] & H_{n-k}(Y; \ZZ) \ar[d, "\underline \phi"]  \\
H^k(Y; \ZZ) \ar[r, "\op{PD}"] & H_{n-k}^{\op{BM}}(Y; \ZZ),
\end{tikzcd}
\end{equation} 
as can be seen by observing that
both horizontal maps are given by capping with the fundamental class $[Y] \in H_n^{\op{BM}}(Y; \ZZ)$.  We immediately deduce the following.
\begin{lemma}\label{l:narPD}
The Poincar\'e duality isomorphism $H^k(Y; \ZZ) \cong H_{n-k}^{\op{BM}}(Y; \ZZ)$ induces an isomorphism
$H^k_{\op{nar}}(Y; \ZZ) \cong H_{n-k}^{\op{nar}}(Y; \ZZ)$.
\end{lemma}

\begin{prop}\label{p:pp}
Let $f: X \to Y$ be a smooth, proper, oriented map between the manifolds $X$ and $Y$.  There exist induced
pullback and pushforward maps 
\begin{align*}
f^*_{\op{nar}}: &H^*_{\op{nar}}(Y; \ZZ) \to H^*_{\op{nar}}(X; \ZZ)\\
f_*^{\op{nar}}:& H^*_{\op{nar}}(X; \ZZ) \to H^*_{\op{nar}}(Y; \ZZ).
\end{align*}
\end{prop}
\begin{proof}
Because $f$ is proper, there is a pullback with compact support  
\[f^*_{\op{c}}: H^*_{\op{c}}(Y; \ZZ) \to H^*_{\op{c}}(X; \ZZ)\] in addition to the usual pullback on cohomology
\[f^*:H^*(Y; \ZZ) \to H^*(X; \ZZ).\]  
The fact that $f^*$ induces a pullback on the narrow subspace follows from the commutative diagram
\[\begin{tikzcd}
H^*_{\op{c}}(Y; \ZZ) \ar[r, "f^*_{\op{c}}"] \ar[d, "\phi"]& H^*_{\op{c}}(X; \ZZ)\ar[d, "\phi"]\\
H^*(Y; \ZZ) \ar[r, "f^*"] & H^*(X; \ZZ).
\end{tikzcd}\]

The pushforward map \[f_*^{\op{c}}: H^*_{\op{c}}(X; \ZZ) \to H^*_{\op{c}}(Y; \ZZ)\] is defined via Poincar\'e duality  together with the pushforward on homology.  
Because $f$ is proper, there is a well defined pushforward map 
\[f_*^{\op{BM}}: H_*^{\op{BM}}(X; \ZZ) \to H_*^{\op{BM}}(Y; \ZZ).\]  Composing with Poincar\'e duality defines a pushforward
\[f_*: H^*(X; \ZZ) \to H^*(Y; \ZZ).\]  The following diagram commutes,
\[\begin{tikzcd}
H_*(X; \ZZ) \ar[r, "f_*"] \ar[d, "\underline \phi"]& H_*(Y; \ZZ)\ar[d, "\underline \phi"]\\
H_*^{\op{BM}}(X; \ZZ) \ar[r, "f_*^{\op{BM}}"] & H_*^{\op{BM}}(Y; \ZZ).
\end{tikzcd}\] 
By invoking Poincar\'e duality and \eqref{e:PD}, we see that $\phi \circ f_*^{\op{c}} = f_*^{\op{c}} \circ \phi$, therefore
$f_*$ induces a pushforward on the narrow subspaces.

\end{proof}

\subsection{Pairing}
To simplify proofs, we will focus on the de Rham definition of narrow cohomology in what follows.
The wedge product on differential forms induces the cup products 
\begin{align} \label{e:prod}
\cup: &H^i(Y; \RR) \times H^j(Y; \RR) \to H^{i + j}(Y; \RR), \\
\cup: &H^i_{\op{c}}(Y; \RR) \times H^j(Y; \RR) \to H^{i + j}_{\op{c}}(Y; \RR). \nonumber
\end{align}
\begin{lemma}\label{l:cupzero}
The cup product is zero when restricted to $\ker(\phi) \times H^*_{\op{nar}}(Y; \RR) \subset H^*_{\op{c}}(Y; \RR) \times H^*(Y; \RR)$.
\end{lemma}
\begin{proof}
Let $\Omega$ denote a closed differential form representing a class $\omega$ in $\ker(\phi)$.  Then by definition of $\ker(\phi)$ there exists a form $\Psi$ (potentially with non-compact support) such that $d \Psi = \Omega$.  If $\Theta$ is a closed form with compact support representing $\theta \in H^*_{\op{nar}}(Y; \RR)$, then
\[ d(\Theta \wedge \Psi) = \Theta \wedge d\Psi = \Theta \wedge \Omega.\] Since the support of $\Theta \wedge \Psi$ is contained in the support of $\Theta$ and is therefore compact, we conclude that 
$\theta \cup \omega = 0 \in H^*_{\op{c}}(Y; \RR)$.
%
\end{proof}
Of course the wedge product also induces a cup product
\[\cup: H^i_{\op{c}}(Y; \RR) \times H^j_{\op{c}}(Y; \RR) \to H^{i + j}_{\op{c}}(Y; \RR).\]
It is clear from the definitions that for $\alpha, \beta \in H^*_{\op{c}}(Y; \RR)$,
\begin{equation}\label{e:phicup}
\alpha \cup \beta = \alpha \cup \phi(\beta).
\end{equation}

We let 
\[\langle - , - \rangle: H^*(Y; \RR) \times H^*_{\op{c}}(Y; \RR) \to \RR\]
denote the pairing defined by 
\[ \langle \alpha , \beta \rangle := \int_Y \alpha \cup \beta.\]
This is well-defined because $\beta$ and therefore $\alpha \cup \beta$ are compactly supported.  It is known to be non-degenerate \cite{BT}.  
\begin{coro}\label{c:pairperp}  With respect to the pairing $\langle - , - \rangle$, $H^*_{\op{nar}}(Y; \RR) = \ker(\phi)^\perp$.
\end{coro}
\begin{proof}
By Lemma~\ref{l:cupzero}, $H^*_{\op{nar}}(Y; \RR) \subseteq \ker(\phi)^\perp$.  However they are the same rank and so must be equal.
\end{proof}

Given  two classes $\alpha$ and $\beta$ in $H^*_{\op{nar}}(Y; \RR)$, \emph{a-priori} the product $\alpha \cup \beta$ lies in $H^*(Y; \RR)$.  It is clear that $\alpha \cup \beta$ in fact lies in $H^*_{\op{nar}}(Y; \RR)$, so the narrow state space inherits a ring structure from $H^*(Y; \RR)$.

We can refine this product to obtain a class
$\alpha \cup_{\op{c}} \beta$ lying in $H^*_{\op{c}}(Y; \RR)$ as follows.
\begin{defi}\label{d:cupcom}
Given $\alpha$ and $\beta$ in $H^*_{\op{nar}}(Y; \RR)$, define the \emph{compactly supported} cup product of $\alpha$ and $\beta$ to be
\[\alpha \cup_{\op{c}} \beta := \tilde \alpha \cup \beta \in H^*_{\op{c}}(Y; \RR),\]
where $\tilde \alpha$ is a lift of $\alpha$.
\end{defi}
\begin{coro}
The product described above is well-defined.
\end{coro}
\begin{proof}
The class $\tilde \alpha$ is well-defined up to a choice of elements in $\ker(\phi)$.  If $\tilde \alpha '$ is a different lift, then 
\[\tilde \alpha '  = \tilde \alpha + \alpha_k\]
where $\alpha_k \in \ker(\phi)$.
By Lemma~\ref{l:cupzero}, $\alpha_k \cup \beta = 0$ in $H^*_{\op{c}}(Y; \RR)$, so 
$\tilde \alpha' \cup \beta = \tilde \alpha \cup \beta$. 
\end{proof}
With the above we can equip $H^*_{\op{nar}}(Y; \RR)$ with a pairing.
\begin{defi}
Given $\alpha$ and $\beta$ in $H^*_{\op{nar}}(Y; \RR)$, define
\[\langle \alpha, \beta \rangle^{\op{nar}} := \int_Y \alpha \cup_{\op{c}} \beta.\]
\end{defi}
\begin{prop}\label{p:nondegpai}
The pairing $\langle - , - \rangle^{\op{nar}}$ is nondegenerate.
\end{prop}

\begin{proof}
Given a nonzero element $\beta \in H^*_{\op{nar}}(Y; \RR)$, there exists an element $\tilde \alpha \in H^*_{\op{c}}(Y; \RR)$ 
which pairs non-trivially with $\beta$.
By definition, 
\[\langle \phi(\tilde \alpha), \beta\rangle^{\op{nar}} =  \int_Y\phi(\tilde  \alpha) \cup_{\op{c}} \beta = \int_Y \tilde \alpha \cup \beta = \langle \tilde \alpha, \beta\rangle \neq 0.\]
\end{proof}

\subsection{Orbifolds}\label{ss:orb}
More generally, 
let $\cY$ be an oriented orbifold in the sense of \cite{ALR}.  Assume $\cY$ admits a finite good cover \cite{ALR}.  Let $I\cY$ denote the \emph{inertia stack} of $\cY$, defined by the fiber diagram
\[
\begin{tikzcd}
I\cY \ar[r]\ar[d] & \cY \ar[d, "\Delta"] \\
\cY \ar[r, "\Delta"] & \cY \times \cY.
\end{tikzcd}
\]
The orbifold $I\cY$ is a disjoint union of connected components called twisted sectors.
These components are indexed by equivalence classes $\{\gamma\}_T$ of isotropy elements.
The connected component indexed by $\gamma$ is referred to as the $\gamma$-twisted sector and will be denoted by $\cY_\gamma$.   There is an involution 
\[I: I\cY \to I\cY\] exchanging twisted sectors with their inverses.
In the case $\cY = [V/\Gamma]$, the involution $I$ maps $\cY_\gamma$ to $\cY_{\gamma^{-1}}$ via the natural isomorphism.  

Let $\cY_\gamma$ be a twisted sector.  Take a point $(y, \gamma) \in \cY_\gamma$.  The tangent space $Ty\cY$ splits as a direct sum of eigenspaces with respect to the action of $\gamma$:
\[ T_y\cY = \bigoplus_{0 \leq f <1} (T_y\cY)_f\]
where $\gamma$ acts on $(T_y\cY)_f$ by multiplication by $e^{2 \pi \sqrt{-1} f}$. Define the \emph{age shift} for $\cY_\gamma$ to be
\[ \iota_\gamma = \sum_{0 \leq f < 1} f \dim_\CC (T_y\cY)_f.\]
\begin{defi}\cite{ChenR1, ALR}
The \emph{Chen--Ruan cohomology} of $\cY$ of (Chen--Ruan) degree $k$ is
\[H^k_{\op{CR}}(\cY):= \bigoplus_{\gamma \in T} H^{k - 2 \iota_\gamma}(\cY_\gamma; \CC),\]  
and 
\[H^*_{\op{CR}}(\cY) := \bigoplus_{k \in \QQ_{\geq 0}} H^k_{\op{CR}}(\cY).\]
Note that we will always use complex coefficients in what follows.  Define the \emph{compactly supported Chen--Ruan cohomology} 
$H^*_{\op{CR}, \op{c}}(\cY)$ analogously.
Define the \emph{Chen--Ruan pairing} 
\[ \langle \alpha, \beta \rangle^\cY :H^*_{\op{CR}}(\cY) \times H^*_{\op{CR}, \op{c}}(\cY) \to \CC\]
to be
\[\langle \alpha, \beta \rangle^\cY := \int_{I\cY} \alpha \cup I_*(\beta).\]

\end{defi}
As in the previous section, let $\phi: H^*_{\op{CR}, \op{c}}(\cY)  \to H^*_{\op{CR}}(\cY)$ denote the natural map.

\begin{defi}
Define the \emph{narrow} subspace of $H^*_{\op{CR}}(\cY)$ to be
\[H^*_{\op{CR}, \op{nar}}(\cY) := \op{im}(\phi).\]
Denote by $\ker(\phi)$ the kernel in $ H^*_{\op{CR}, \op{c}}(\cY)$.
Given $\alpha$ and $\beta$ in $H^*_{\op{CR}, \op{nar}}(\cY)$, define
\[ \langle \alpha, \beta\rangle^{\cY, \op{nar}} := \int_{I\cY} \alpha \cup_{\op{c}} I_*(\beta).\]
\end{defi}
By orbifold versions of the same arguments, Lemma~\ref{l:narPD},  Proposition~\ref{p:pp}, Corollary~\ref{c:pairperp}, and Proposition~\ref{p:nondegpai} also hold for the narrow Chen--Ruan cohomology of an orbifold.
%

\begin{rema}\label{remaCR}
The correct cup product to use in this setting is the \emph{Chen--Ruan cup product} \cite{ChenR1}.  Unless otherwise specified, all products of classes in Chen--Ruan cohomology in this paper are understood to be Chen--Ruan cup products.
The definition will be given in Definition~\ref{CRprod}.  For now we simply note the following fact \cite[Lemma~2.3.8]{Ts}.
For $\alpha \in H^*(\cY) \subset H^*_{\op{CR}}(\cY)$ supported in the untwisted sector and any $\beta \in  H^*_{\op{CR}}(\cY)$,
\[\alpha \cup \beta = q*(\alpha) \cup_{I\cY} \beta \in H^*_{\op{CR}}(\cY),\]
where $q: I\cY\to \cY$ is the natural map and $\cup_{I\cY}$ denotes the usual cup product on $H^*(I\cY)$.  As a consequence of the above fact, the fundamental class on the untwisted sector $1 \in H^*(\cY) \subset H^*_{\op{CR}}(\cY)$ is the multiplicative  identity with respect to the Chen--Ruan cup product.
\end{rema}

\subsection{The total space example}
Let $\cX$ be a compact oriented orbifold of dimension $n$ and let $\cE \to \cX$ be a complex vector bundle of complex rank $r$ over $\cX$.  
Let $\cY$ denote the total space of $\cE^\vee$ over $\cX$ (we use $\cE^\vee$ only to be consistent with later sections), endowed with the orientation induced by the orientation on $\cX$.  Let $i: I\cX \to I\cY$ denote the inclusion via the zero section.  This is a proper, oriented map.   Let $\pi: I\cY \to I\cX$ denote the projection.
\begin{prop}\label{p:tot}  The following are equal:
\[H^*_{\op{CR}, \op{nar}}(\cY) = \op{im}\left(i_*: H^*_{\op{CR}}(\cX) \to H^*_{\op{CR}, \op{nar}}(\cY) \right).
\]
If $\cE$ is  pulled back from a vector bundle $E \to X$ on the coarse space of $\cX$, then this is equal to 
\[\op{im}\left( e(\pi^*\cE^\vee) \cup - \right)\]
where $ e(\pi^*\cE^\vee) \cup - : H^*_{\op{CR}}(\cY) \to H^*_{\op{CR}}(\cY)$ is the (Chen-Ruan) cup product with the Euler class of $\pi^*\cE^\vee \in H^*(\cY) \subset H^*_{\op{CR}}(\cY)$.
\end{prop}
\begin{proof}
By the proof of Proposition~\ref{p:pp}, $\phi \circ i_*^{\op{c}} = i_* \circ \phi = i_*$, so \[\op{im}\left(i_*: H^*_{\op{CR}}(\cX) \to H^*_{\op{CR}, \op{nar}}(\cY) \right) \subseteq H^*_{\op{CR}, \op{nar}}(\cY).\]  The two are equal if $i^{\op{c}}_*: H^*_{\op{CR}}(\cX) \to H^*_{\op{CR}, \op{c}}(\cY)$ is an isomorphism.  This holds because the pushforward $i_*: H_*(I\cX; \RR) \to H_*(I\cY; \RR)$ in homology is an isomorphism.

Assume that $\cE$ is  pulled back from a vector bundle $E \to X$ on the coarse space of $\cX$.  In this case all isotropy groups of $\cX$ act trivially on the fibers of $\cE$.  Consequently each component $\cY_\gamma$ of $I\cY$ is given by the total space of $\cE^\vee$ restricted to $\cX_\gamma$.
 The base $\cX_\gamma \subset \cY_\gamma$ can be viewed as the zero locus of the tautological section of $\pi^*\cE^\vee$ restricted to $\cY_\gamma$.  On the untwisted sector, we have
$e(\pi^*\cE^\vee) = i_*(1) \in H^*(\cY)\subset H^*_{\op{CR}}(\cY)$.  Then for all $\alpha \in H^*_{\op{CR}}(\cY)$,
\[e(\pi^*\cE^\vee) \cup \alpha = i_*(1) \cup \alpha = i_*(1 \cup i^*(\alpha)) = i_*(i^*(\alpha))\]
where the second equality follows from the projection formula and Remark~\ref{remaCR} and the last equality is Remark \ref{remaCR}.  The proof concludes by noting that $i^*: H^*_{\op{CR}}(\cY) \to H^*(\cX)$ is an isomorphism.
%
\end{proof}

\section{Quantum $D$-modules for a proper target}\label{s:dmod}
This section serves to recall the basic definitions and constructions of Gromov--Witten theory and to set notation.  We recall how the genus zero Gromov--Witten theory of a smooth and proper Deligne--Mumford stack $\cX$ defines a flat connection known as the Dubrovin connection, which in turn gives a \emph{quantum $D$-module} with integral structure.   See \cite{CK, Iri, Iri3} for more details.

For the remainder of the paper, $\cX$ will be a smooth Deligne--Mumford stack over $\CC$.  We will use the same notation to denote the corresponding complex orbifold.  
We always assume further that the coarse moduli space of $\cX$, denoted by $X$, is projective.
\begin{defi}
Given $\cX$ as above, let $\sMbar_{h,n}(\cX, d)$ denote the moduli space of representable degree $d$ stable maps from orbi-curves of genus $h$ with $n$ marked points (as defined in \cite{AGV}).  Here $d$ is an element of 
 the cone $\op{Eff} = \op{Eff}(\cX) \subset H_2(\cX, \QQ)$ of effective curve classes.   Denote by $[\sMbar_{h,n}(\cX, d)]^{vir}$ the virtual fundamental class of \cite{BF} and \cite{AGV}. 
\end{defi}

Recall that for each marked point $p_i$ there exist evaluation maps 
\[ev_i: \sMbar_{h, n}(\cX, d)  \to \bar I \cX,\] where  $\bar I \cX$ is the \emph{rigidified inertia stack} as in \cite{AGV}.  By the discussion in Section 6.1.3 of \cite{AGV}, it is convenient to work as if the map $ev_i$ factors through $I \cX$.  While this is not in fact true, due to the isomorphism $H^*(\bar I \cX; \CC) \cong H^*(I\cX; \CC)$ it makes no difference in terms of defining Gromov--Witten invariants.  

%

\begin{defi}
Given $\alpha_1, \ldots, \alpha_n \in H^*_{\op{CR}}(\cX)$, $d \in \op{Eff}$ and integers $h, b_1, \ldots, b_n \geq 0$, define the Gromov--Witten invariant
\[ \langle \alpha_1 \psi^{b_1}, \ldots, \alpha_n \psi^{b_n} \rangle^{\cX}_{h, n, d} := 
\int_{[\sMbar_{h,n}(\cX, d)]^{vir}} \prod_{i=1}^n ev_i^*( \alpha_i) \cup \psi_i^{b_i}.\]
\end{defi}

\subsection{Twisted invariants}\label{s:td}
Many of the results of \S~\ref{s:twisting} and \S~\ref{s:qsd} are in terms of \emph{twisted invariants}, which we define below.  See \cite{CG} for details of the theory.

Let $\cE$ be a vector bundle on a  Deligne--Mumford stack $\cX$.  
Given formal parameters $s_k$ for $k \geq 0$, one defines the formal invertible multiplicative characteristic class
\[
\bs:  \cE \mapsto \exp \left(  \sum_{k \geq 0} s_k \ch_k (\cE) \right).\]
The twisted Gromov--Witten invariants depend on the parameters $\bs = (s_0, s_1, \ldots )$ and take values in $\CC[[\bs]]$.  Let $f^*(\cE)$ denote the pullback of $\cE$ to the universal curve $\widetilde{\cc C}$ over $\sMbar_{h,r}(\cX, d)$.

\begin{defi}
Given $\alpha_1, \ldots, \alpha_n$ in $H^*_{\op{CR}}(\cX)$, $d \in \op{Eff}$ and integers $h, b_1, \ldots, b_n \geq 0$, define the \emph{$\bs$-twisted} Gromov--Witten invariant of $\cX$ to be
\begin{align*} &\langle \alpha_1 \psi^{b_1}, \ldots, \alpha_n \psi^{b_n} \rangle^{\cX, \bs}_{h, n, d} \\ := &
\int_{[\sMbar_{h,n}(\cX, d)]^{vir}} \exp \left( \sum_{k \geq 0} s_k \ch_k (\mathbb R \pi_* f^*(\cE)) \right)  \prod_{i=1}^n ev_i^*( \alpha_i) \cup \psi_i^{b_i}.\end{align*}
\end{defi}

\begin{defi}
Define an \emph{$\bs$-twisted Gromov--Witten pairing} as follows.  Given
$\alpha, \beta \in H^*_{\op{CR}}(\cX) \otimes \CC[[\bs]]$, define
\[ \langle \alpha, \beta \rangle^{ \cX, \bs} := \br{\exp \left( \sum_{k \geq 0} s_k \ch_k (\cE) \right) \alpha, \beta }^\cX.\]
\end{defi}

\subsection{Quantum connections}

Choose a basis $\{T_i\}_{i \in I}$ for the $H^*_{\op{CR}}(\cX)$ state space such that $I = I' \coprod I''$ where $I''$ indexes a basis for the degree two part of the cohomology supported on the untwisted sector, and $I'$ indexes a basis for the cohomology of the twisted sectors together with the degree not equal to two cohomology of the untwisted sector.    
Let $\bt ' = \sum_{i \in I '} t^i T_i$ and let $\bt = \sum_{i \in I' \cup I''} t^i T_i$.  Let $q^i = e^{t_i}$ for $i \in I''$.  Choose $\alpha_1, \ldots, \alpha_n$ from $H^*_{\op{CR}}(\cX) $.
 For $\square = \cX$ or $\square = (\cX, \bs)$, 
\begin{equation}\label{e:fps20}\langle \langle \alpha_1 \psi^{b_1}, \ldots, \alpha_n \psi^{b_n} \rangle \rangle^\square(\bt) := \sum_{d \in \op{Eff}}\sum_{k\geq 0}\frac{1}{k!} \langle \alpha_1 \psi^{b_1}, \ldots, \alpha_n \psi^{b_n}, \bt, \ldots, \bt \rangle^\square_{0,n+k,d}\end{equation}
where a summand is implicitly assumed to be zero if $d = 0$ and $n+k < 3$.  
Let $\CC[[\bt ']] := \CC[[ t^i]]_{i \in I ' }$ and $\CC[[\bq]] := \CC[[q^i]]_{i \in I ''}$.
By the divisor equation \cite[Section~10.1]{CK}, if $b_1 = \cdots = b_n = 0$ then \eqref{e:fps20} can be viewed as a formal power series in $\CC[[\bt ', \bq]]$
 or $\CC[[\bt ', \bq, \bs]]$ in the twisted case.  
 
 \begin{nota}\label{n:P}
 Denote by $P^\square$  the power series $\CC[[\bt ', \bq]]$ when $\square = \cX$ refers to Gromov--Witten theory, or $\CC[[\bt ', \bq, \bs]]$ when $\square = (\cX, \bs)$ refers to a $\bs$-twisted theory.
 \end{nota}

\begin{defi}
For elements $\alpha, \beta$ in $H^*_{\op{CR}}(\cX)$ with $\square = \cX$ or $(\cX, \bs)$, define the \emph{quantum product}
$\alpha \bullet_{\bt}^\square \beta \in H^*_{\op{CR}}(\cX) \otimes P^\square$ by the formula
\[ \langle \alpha \bullet_{\bt} \beta, \gamma \rangle^\square = \langle \langle \alpha, \beta, \gamma\rangle \rangle^\square(\bt)\]
for all $\gamma \in H^*_{\op{CR}}(\cX)$.  
\end{defi}
In the non-twisted setting ($\bs = 0$) this is equivalent to
\[\alpha \bullet_{\bt}^\square \beta := \sum_{d \in \op{Eff}} \sum_{k \geq 0} \frac{1}{k!} I_* \circ {ev_3}_* \left( ev_1^*(\alpha) \cup ev_2^*(\beta) \cup \prod_{j=4}^{k+3} ev_j^*(\bt) \cap \left[ \sMbar_{0,k+3}(\cc X, d)\right]^{vir}
\right)\]
 where to simplify notation we suppress the Poincar\'e duality isomorphism $\op{PD}: H^*_{CR}(\cX)\to H_*^{\op{BM}}(I\cX) $ from \eqref{e:PD} which identifies the vector spaces on the left and right side of the expression. 

\begin{defi}\label{CRprod}
For elements $\alpha, \beta$ in $H^*_{\op{CR}}(\cX)$  define the \emph{Chen--Ruan cup product}
$\alpha \cup \beta \in H^*_{\op{CR}}(\cX)$ by the formula
\[ \langle \alpha \cup \beta, \gamma \rangle^\cX = \langle \langle \alpha, \beta, \gamma\rangle \rangle^\cX(\bf{0})|_{\bq = 0}\]
for all $\gamma \in H^*_{\op{CR}}(\cX)$.  

\end{defi}

\subsection{Quantum $D$-module}

Following \cite{Iri3}, define a $z$-sesquilinear pairing $S^\square$ on $H^*_{\op{CR}}(\cX) \otimes P^\square[z, z^{-1}]$ as
\[ S^\square( u(z), v(z)) := (2 \pi \sqrt{-1} z)^{\dim(\cX)} \langle u(-z), v(z) \rangle^\square.\] 

\begin{defi}\label{d:dub1} The Dubrovin connection is given by the formula 
\[ \nabla_i^\square  = \partial_i + \frac{1}{z} T_i \bullet_{\bt}^\square,\]
where recall that $\{T_i\}_{i \in I}$ is a basis for $H^*_{\op{CR}}(\cX)$.  These operators act on  $H^*_{\op{CR}}(\cX) \otimes P^\square[z, z^{-1}]$.
\end{defi}

When we are not in 
the twisted setting, we can define the quantum connection in the $z$-direction as well.\footnote{The definition of $\nabla_z^\square$ is based on the fact that (when we are not in the twisted setting) the virtual class is of pure dimension on connected components.}  Define the Euler vector field
\[ \mathfrak E := \partial_{\rho(\cX)}+ \sum_{i \in I} \left(1 - \frac{1}{2} \deg T_i\right) t^i \partial_i\]
where $\rho(\cX):= c_1(T\cX) \subset H^2(\cX; \CC)$.

Define the grading operator $\op{Gr}$ by 
\[\op{Gr}(\alpha) := \frac{\deg \alpha}{2} \alpha\] for $\alpha$ in $H^*_{\op{CR}}(\cX)$.
Then define
\[ \nabla_{z}^\cX = \partial_z - \frac{1}{z^2} \mathfrak E \bullet^\cX_{\bt} + \frac{1}{z} \op{Gr}.\]



Define the operator $L^\square(\bt, z)$ 
 by
 \[L^\square(\bt, z) \alpha := \alpha + \sum_{i \in I} \br{\br{ \frac{\alpha}{-z - \psi}, T^i }}^\square(\bt) T_i.\]
 When we are not in the twisted setting this is equivalent to
\begin{equation}\label{e:Lgen}
L^{\cX}(\bt, z)(\alpha) := \alpha + \sum_{d \in \op{Eff}} \sum_{k \geq 0} \frac{1}{k!}
I_* \circ{ev_2}_* \left(\frac{ev_1^*(\alpha)}{-z - \psi} \prod_{j=3}^{k+2} ev_j^*(\bt) \cap \left[\sMbar_{0, k+2}(\cX, d)\right]^{vir} \right).
\end{equation}

\begin{prop}\cite{CK, Iri3} \label{p:flat} Let $\square = \cX$ or $(\cX, \bs)$.  
The quantum connection $\nabla^\square$ is flat.
 In the untwisted case,
\begin{equation} \nabla^\cX_i \left(L^\cX(\bt, z) z^{- \op{Gr}}z^{\rho(\cX)} \alpha\right) = \nabla^\cX_z \left(L^\cX(\bt, z) z^{- \op{Gr}} z^{\rho(\cX)}\alpha\right) = 0 \end{equation}  for  $i \in I$ and $\alpha \in H^*_{\op{CR}}(\cX)$.  
In  the twisted case,
\begin{equation} \nabla^{\cX, \bs}_i \left(L^{\cX, \bs}(\bt, z)  \alpha\right) = 0.\end{equation}
In both the twisted and untwisted setting the pairing $S^\square$ is flat with respect to $\nabla^\square$.
For $\alpha, \beta \in H^*_{\op{CR}}(\cX)$,
\begin{equation}\label{e:invadj} \langle L^\square(\bt, -z) \alpha, L^\square(\bt, z) \beta \rangle^\square = \langle \alpha, \beta\rangle^\square.\end{equation}
\end{prop}

\begin{defi}\label{d:qdm}
The \emph{quantum $D$-module} for $\cX$,  QDM($\cX$), is defined to be the triple:
\[QDM(\cX) := \left(H^*_{\op{CR}}(\cX) \otimes P^\cX[z, z^{-1}], \nabla^\cX, S^\cX \right).\]
\end{defi}

\subsubsection{Ambient Gromov--Witten theory}\label{ss:amb}
There is a restricted quantum $D$-module for local complete intersection sub-stacks.

\begin{defi}
A vector bundle $\cE \to \cX$ is called \emph{convex} if, for every representable morphism 
\[ f: \cC \to \cX\]
from a genus zero orbi-curve $\cC$, the cohomology $H^1(\cC, f^*(\cE))$ is zero.
\end{defi}
\begin{rema}\label{r:conv}
Note that convexity of $\cE$ implies that it is pulled back from a vector bundle $E \to X$ on the coarse space \cite[Remark~5.3]{CGIJJM}.  Furthermore one can  check that $E$ will itself be convex in this case.
\end{rema}

Let $\cX$ be as before. Consider $\cE$ a convex vector bundle on $\cX$, and let $\cZ$ be the zero locus of a transverse section $s \in \Gamma( \cX, \cE)$.  Let $j: \cZ \to \cX$ denote the inclusion map and define $H^*_{\op{CR}, \op{amb}}(\cZ) := \op{im} (j^*)$.
\begin{assumption}\label{a:ass2}
As in \cite{IMM}, we will always assume that the Poincar\'e pairing on $H^*_{\op{CR}, \op{amb}}(\cZ)$ is non-degenerate.  This is equivalent to the condition that
\begin{equation}\label{e:ass2} H^*_{\op{CR}}(\cZ) = \op{im}(j^*) \oplus \ker(j_*).\end{equation}
\end{assumption}
Assumption~\ref{a:ass2} holds for instance if $\cE$ is the pullback of an ample line bundle on $X$ and $\cZ$ intersects each twisted component of $\cX$ transversally.
\begin{prop}\cite[Corollary 2.5]{Iri3}
For $\bar \bt \in H^*_{\op{CR}, \op{amb}}(\cZ)$, the quantum product $\bullet^\cZ_{\bar \bt}$ is closed on $H^*_{\op{CR}, \op{amb}}(\cZ)$.  The quantum connection and solution $L^\cZ(\bar \bt, z)$ preserve $H^*_{\op{CR}, \op{amb}}(\cZ)$ for $\bar \bt \in H^*_{\op{CR}, \op{amb}}(\cZ)$.
\end{prop}
\begin{rema}
The proof of the above proposition in \cite[Corollary 2.5]{Iri3} follows an argument from \cite{Pan}.  The assumptions on $\cX$ and $\cE$ above are weaker than in either, but the same argument goes through, as observed in  \cite[Remark 2.2]{Coa}.
\end{rema}
%
\begin{defi}
The \emph{ambient} quantum $D$-module is defined to be 
\[ QDM_{\op{amb}}(\cZ) := (H^*_{\op{CR}, \op{amb}}(\cZ) \otimes P^{\cZ, \op{amb}}[z, z^{-1}], \nabla^\cZ, S^\cZ)\]
where $P^{\cZ, \op{amb}}$ denotes the restriction of $P^\cZ$ to $H^*_{\op{CR}, \op{amb}}(\cZ)$.
\end{defi}

\subsection{Integral structure}\label{ss:is}

In \cite{Iri}, Iritani defines an integral structure for Gromov--Witten theory. We recall the ingredients here.

For $\cc F$ a vector bundle on $\cX$, let $\cc F_\gamma := \cc F|_{\cX_\gamma}$ denote the restriction of $\cc F$ to a twisted sector $\cX_\gamma$.
Recall as in \S~\ref{ss:orb}, that  $\cc F_\gamma$ splits into a sum of eigenbundles
\[\cc F_\gamma  = \bigoplus_{0 \leq f <1} \cc F_{\gamma, f}\] 
where the action of $\gamma$ on $\cc F_{\gamma, f}$ is multiplication by $e^{2 \pi \sqrt{-1} f}$.

\begin{defi}\label{d:Gamma}
Define the \emph{Gamma class} $\hat \Gamma (\cc F)$ to be the class in $H^*(I\cX)$:
\[ \hat \Gamma (\cc F)  := \bigoplus_{\gamma \in T} \prod_{0 \leq f < 1} \prod_{i = 1}^{\op{rk}(\cc F_{b,f})} \Gamma(1 - f + \rho_{b,f,i})\]
where $\Gamma(1 + x)$ should be understood in terms of its Taylor expansion at $x = 0$,  and $\{\rho_{b,f,i}\}$ are the Chern roots of $\cc F_{b,f}$.  Define $\hat \Gamma_\cX$ to be the class $\hat \Gamma(T\cX)$.
\end{defi}

\subsubsection{Flat sections}\label{ss:fs}
\begin{defi}
Define the operator $\deg_0$ to be the degree operator which multiplies a homogeneous class by its unshifted degree.  In Gromov--Witten theory it multiplies a class in $H^n(I\cX)$ (with the standard grading) by $n$.
\end{defi}

Denote by $D(\cX) := D^b(\cX)$ the bounded derived category of coherent sheaves on $\cX$.  We will omit the superscript $b$.
Given an object $F$ in $D(\cX)$,  define $\bs^\cX(F)(\bt, z)$ to be
\[ \frac{1}{(2 \pi \sqrt{-1})^{\dim(\cX)}} L^\cX(\bt, z) z^{-Gr} z^{\rho(\cX)} \hat \Gamma_{\cX} \cup_{I\cX}\left( (2 \pi \sqrt{-1})^{\deg_0/2} I^*(\ch(F))\right),\]
where $\cup_{I\cX}$ denotes the ordinary cup product on $H^*(I\cX)$.
\begin{prop}\cite{Iri3}\label{p:pairing}
The map $\bs^\cX$ identifies the pairing in the derived category with $S^\cX$:
\[ S^\cX (\bs^\cX(F)(\bt, z), \bs^\cX(F ')(\bt, z)) = 
e^{ \pi \sqrt{-1} \dim(\cX)}
\chi(F ', F).\]
\end{prop}

\begin{assumption}\label{as2}
Assume that $H^*_{\op{CR}}(\cX)$ is spanned by the image of \[\ch: D(\cX) \to H^*_{\op{CR}}(\cX).\]
\end{assumption}
The set 
\[ \{ \bs^\cX(F)(\bt, z) | F \in D(\cX)\}\] forms a lattice in $\op{ker}( \nabla^\cX)$.  This is the integral structure of the quantum $D$-module $QDM(\cX)$.

Let $j: \cZ \to \cX$ be a smooth subvariety defined as in Section~\ref{ss:amb} (i.e. by a convex vector bundle).  Then again with assumption~\ref{as2} we can define an integral structure for the ambient quantum $D$-module as \[ \{ \bs^\cZ(F)(\bt, z) | F \in j^*(D(\cX))\}.\]
Note that by orbifold Grothendieck--Riemann--Roch, $\ch(F)$ will lie in $H^*_{\op{CR}, \op{amb}}(\cZ)$ for $F \in j^*(D(\cX))$.

%

\section{Quantum $D$-modules for a non-proper target}
In fact most of the constructions of genus zero Gromov--Witten theory go through in the case of a non-proper target $\cY$.  We explore this in this section.  The majority of this section is known to the experts, see e.g. \cite{GP, IMM}.  The perspective of \S~\ref{s:nar} in terms narrow cohomology, however, is new.

Let $\cY$ denote a smooth Deligne--Mumford stack with quasi-projective coarse moduli space $Y$.
Although $\cY$ may not be proper, if  the evaluation maps 
\[ev_i: \sMbar_{h, n}(\cY, d)  \to \bar I \cY\] are proper, 
one can still define Gromov--Witten invariants in many cases.  In particular, there is still a well-defined virtual class $[\sMbar_{h,n}(\cY, d)]^{vir}$ over which one can integrate classes of compact support.
\begin{defi}
Assume the maps $ev_i: \sMbar_{h, n}(\cY, d)  \to \bar I \cY$ are proper for $1 \leq i \leq n$.  Given $\alpha_1 \in H^*_{\op{CR}, \op{c}}(\cY)$, $\alpha_2, \ldots, \alpha_n \in H^*_{\op{CR}}( \cY)$ and integers $b_1, \ldots, b_n \geq 0$ define the Gromov--Witten invariant
\[ \langle \alpha_1 \psi^{b_1}, \ldots, \alpha_n \psi^{b_n} \rangle^{\cY}_{h, n, d} := 
\int_{[\sMbar_{h,n}(\cY, d)]^{vir}} \prod_{i=1}^n ev_i^*( \alpha_i) \cup \psi_i^{b_i}.\]
Given $\alpha_1, \alpha_2 \in H^*_{\op{CR}, \op{nar}}(\cY)$, $\alpha_3, \ldots, \alpha_n \in H^*_{\op{CR}}( \cY)$ and integers $b_1, \ldots, b_n \geq 0$ define the Gromov--Witten invariant
\[ \langle \alpha_1 \psi^{b_1}, \ldots, \alpha_n \psi^{b_n} \rangle^{\cY}_{h, n, d} := 
\int_{[\sMbar_{h,n}(\cY, d)]^{vir}} \alpha_1 \cup_{\op{c}} \alpha_2 \cup \prod_{i=3}^n ev_i^*( \alpha_i) \cup \psi_i^{b_i}\]
where recall the product $\cup_{\op{c}}: H^*_{\op{CR}, \op{nar}}(\cY) \times H^*_{\op{CR}, \op{nar}}(\cY) \to H^*_{\op{CR}, \op{c}}(\cY)$ was defined via a lift $\tilde \alpha_1 \in H^*_{\op{CR}, \op{c}}(\cY)$ as in Definition~\ref{d:cupcom}.
\end{defi}


The Lemma below gives an important scenario when the evaluation maps are in fact proper.
Let $\cX$ be  a smooth proper Deligne--Mumford stack, and  $\cE \to \cX$ a vector bundle on $\cX$.  Let $\cY$ denote the total space of $\cE^\vee$ over $\cX$.  
\begin{lemma}\label{l:propmap}
For $1 \leq i \leq n$, the evaluation map $ev_i: \sMbar_{0, n}(\cY, d)  \to \bar I \cY$ is proper in the following situations:
%
\begin{enumerate}
\item The degree $d = 0$; 
\item The vector bundle $\cE$ is convex. 

\end{enumerate}
\end{lemma}

\begin{proof}
In the first case, the evaluation map $ev_i: \sMbar_{0, n}(\cY, 0) \to \bar I \cY$ factors through a rigidification of $I^n \cY = \coprod_{\bf{g}} \cY_{g_1, \ldots, g_n}$, the stack of $n$-multisectors \cite[Example~2.5]{ALR}.  The map $\pi_1: \sMbar_{0, n}(\cY, 0) \to \bar I^n \cY$ forgets the source curve while remembering the maps $B \mu_{r_i} \to \cY$ at each marked point $p_i$.  The map $\overline{ev_i}: \bar I^n \cY \to \bar I \cY$ sends the component 
$\cY_{g_1, \ldots, g_n}$ to $\cY_{g_i}$.  The evaluation map $ev_i$ is the composition of the proper map $\pi_1$ and the closed immersion $\overline{ev_i}$ and is therefore proper.

In the second case, we note that a map $f_\cY: \cC \to \cY$ consists of a map to $f: \cC \to \cX$ together with a section $s \in H^0( \cC, f^*(\cE^\vee))$.  
Let $r_\cC: \cC \to C$ be the map to the underlying coarse curve.  Let $\bar i$ be the map $\bar I \cX \to X$.
By Theorem 1.4.1 of \cite{AV}, the composition $\cC \to \cX \to X$ factors through a map $|f|: C \to X$.  
By Remark~\ref{r:conv}, $\cE$ is the pullback of a  vector bundle $E \to X$ and $E$ is convex.
Therefore on each irreducible component $C_j$ of $C$, $E \cong \sum_{i=1}^r \cc O(k_i)$ with each $k_i \geq 0$.  
Therefore the evaluation map 
\[ev_i: H^0(\cC, f^*(\cE^\vee)) = H^0(C, |f|^*(E^\vee)) \to E^\vee|_{|f|(r_\cC(x_i))} = \bar i^* E^\vee|_{f(x_i)}\]
is an injection (see Lemma 3.8 of \cite{IMM}). 
Then
$\sMbar_{0, n}(\cY, d)$ is seen to be a substack of $\sMbar_{0, n}(\cX, d) {}_{ev_i} \times_{proj} \tot( \bar i^* E^\vee)$ via the map $ev_i$.  Here $proj$ is the projection from $\tot(\bar i^* E^\vee)$ to $\bar I\cX$.
%
 In fact we have the following fiber square
\[
\begin{tikzcd}
\sMbar_{0, n}(\cY, d) \ar[d] \ar[r, "\op{inc}"] & \sMbar_{0, n}(\cX, d) {}_{ev_i} \times_{proj} \tot( \bar i^* E^\vee) \ar[d] \\
\sMbar_{0, n}(Y, d) \ar[r] & \sMbar_{0, n}(X, d) {}_{ev_i} \times_{proj} \tot(E^\vee).
\end{tikzcd}
\]
The bottom map is proper by Lemma~3.8 of \cite{IMM}, and therefore so is the top map.  The map
\[ \sMbar_{0, n}(\cX, d) {}_{ev_i} \times_{proj} \tot (\bar i^* E^\vee) \xrightarrow{(ev_i, \op{id}_2)} \bar I\cX 
{}_{\op{id}} \times_{proj} \tot( \bar i^* E^\vee) = \bar I \cY\]
is proper.  The evaluation map
$ev_i: \sMbar_{0, n}(\cY, d) \to \bar I \cY$ is equal to the composition $(ev_i, \op{id}_2) \circ \op{inc}$.

\end{proof}

\begin{rema}
More generally, for a noncompact space $\cY$, the evaluation map $ev_i: \sMbar_{g, n}(\cY, d) \to \bar I \cY$ will be proper if $\cY$ is projective over an affine variety \cite{RoSho}.  However we will not use this in what follows so we omit the proof.
\end{rema}

\subsection{Quantum connections}
Let us assume from now on that $\cY$ is not necessarily proper, but  that the (genus zero) evaluation maps are proper.  

As in  \S~\ref{s:dmod} we may define double brackets.
The setup is as before.
Choose a basis $\{T_i\}_{i \in I}$ for the $H^*_{\op{CR}}(\cY)$ state space such that $I = I' \coprod I''$ where $I''$ indexes a basis for the degree two part of the cohomology supported on the untwisted sector, and $I'$ indexes a basis for the cohomology of the twisted sectors together with the degree not equal to two cohomology of the untwisted sector.   
Let $\bt ' = \sum_{i \in I '} t^i T_i$ and let $\bt = \sum_{i \in I' \cup I''} t^i T_i$.  Let $q^i = e^{t_i}$ for $i \in I''$.  Denote by $P^\cY$ the power series $\CC[[\bt ', \bq]]$.
Choose $\alpha_1, \ldots, \alpha_n$ from $H^*_{\op{CR}}(\cY) \cup H^*_{\op{CR}, \op{c}}(\cY)$.  Assume that for at least one $i$, $\alpha_i$ is in $H^*_{\op{CR}, \op{c}}(\cY)$, or that for some $i < j$, $\alpha_i, \alpha_j \in H^*_{\op{CR}, \op{nar}}(\cY)$.
 Then define
\begin{equation}\label{e:fps2}\langle \langle \alpha_1 \psi^{b_1}, \ldots, \alpha_n \psi^{b_n} \rangle \rangle^\cY(\bt) := \sum_{d \in \op{Eff}}\sum_{k\geq 0}\frac{1}{k!} \langle \alpha_1 \psi^{b_1}, \ldots, \alpha_n \psi^{b_n}, \bt, \ldots, \bt \rangle^\cY_{0,n+k,d}\end{equation}
where a summand is implicitly assumed to be zero if $d = 0$ and $n+k < 3$.  
In the absence of $\psi$-classes this yields a formal power series in $P^\cY$.
 
In the case of a non-proper target, there are two possible quantum products, in analogy with \eqref{e:prod}. 
For elements $\alpha, \beta$ in $H^*_{\op{CR}}(\cY)$, define $\alpha \bullet_{\bt}^\cY \beta \in H^*_{\op{CR}}(\cY) \otimes P^\cY$ by the formula
\[ \langle \alpha \bullet_{\bt} \beta, \gamma \rangle^\cY = \langle \langle \alpha, \beta, \gamma\rangle \rangle^\cY(\bt)\]
for all $\gamma \in H^*_{\op{CR}, \op{c}}(\cY)$.  Similarly to the cup product, we can also multiply a cohomology class with a cohomology class with compact support.  If $\alpha \in H^*_{\op{CR}}(\cY)$ and $\beta \in H^*_{\op{CR}, \op{c}}(\cY)$, define $\alpha \bullet_{\bt}^\cY \beta \in H^*_{\op{CR}, \op{c}}(\cY)$ by the formula
\[ \langle \gamma,  \alpha \bullet_{\bt} \beta \rangle^\cY = \langle \langle \alpha, \beta, \gamma\rangle \rangle^\cY(\bt)\]
for all $\gamma \in H^*_{\op{CR}}(\cY)$.  
Alternatively, in the first case the above definition is equivalent to
\begin{equation}\label{e:noncompprod}\alpha \bullet_{\bt}^\cY \beta := \sum_{d \in \op{Eff}} \sum_{k \geq 0} \frac{1}{k!} I_* \circ{ev_3}_* \left( ev_1^*(\alpha) \cup ev_2^*(\beta) \cup \prod_{j=4}^{k+3} ev_j^*(\bt) \cap \left[ \sMbar_{0,k+3}(\cY, d)\right]^{vir}
\right).\end{equation}
The second case is the same but replacing $ev_2^*$ and ${ev_3}_*$ with ${ev_2}_{\op{c}}^*$ and ${ev_3}^{\op{c}}_*$, the pullback and pushforward in cohomology with compact support.


Exactly as in the proper case, 
The pairing $\langle - , - \rangle^\cY$ can be extended to a $z$-sesquilinear pairing $S^\cY$ between $H^*_{\op{CR}}(\cY) \otimes P^\cY[z, z^{-1}]$ and $H^*_{\op{CR}, \op{c}}(\cY) \otimes P^\cY[z, z^{-1}]$ by defining
\[ S^\cY( u(z), v(z)) := (2 \pi \sqrt{-1} z)^{\dim(\cY)} \langle u(-z), v(z) \rangle^\cY.\]

\begin{defi}
As before, the Dubrovin connection is defined by the formulas
\[ \nabla_i^\cY  = \partial_i + \frac{1}{z} T_i \bullet_{\bt}^\cY\]
and
\[\nabla_{z}^\cY = \partial_z - \frac{1}{z^2} \mathfrak E \bullet^\cY_{\bt} + \frac{1}{z} \op{Gr}\]
where recall that $\{T_i\}_{i \in I}$ is a basis for $H^*_{\op{CR}}(\cY)$.  These operators act on both $H^*_{\op{CR}}(\cY) \otimes P^\cY[z, z^{-1}]$ and $H^*_{\op{CR}, \op{c}}(\cY) \otimes P^\cY[z, z^{-1}]$.  In particular, for $\alpha \in H^*_{\op{CR}, \op{c}}(\cY) \otimes P^\cY[z, z^{-1}]$, $T_i \bullet_{\bt}^\cY \alpha$ and therefore $\nabla_i^\cY \alpha$ lies in $H^*_{\op{CR}, \op{c}}(\cY) \otimes P^\cY[z, z^{-1}]$.
To avoid confusion, we will denote the connection by $\nabla^{\cY}$ when acting on cohomology and by $\nabla^{\cY, \op{c}}$ when acting on cohomology with compact support.
\end{defi}

%
%
%
%
%
%

Define $L^\cY(\bt, z)$ 
by
\begin{equation}\label{e:Lgen0}
L^{\cY}(\bt, z)(\alpha) := \alpha + \sum_{d \in \op{Eff}} \sum_{k \geq 0} \frac{1}{k!}
I_* \circ{ev_2}_* \left(\frac{ev_1^*(\alpha)}{-z - \psi} \prod_{j=3}^{k+2} ev_j^*(\bt) \cap \left[\sMbar_{0, k+2}(\cY, d)\right]^{vir} \right)
\end{equation}
for $\alpha $ in  $H^*_{\op{CR}}(\cY)$.  
Define
 $L^{\cY, \op{c}}(\bt, z)$ 
by
\begin{equation}\label{e:Lgen0c}
L^{\cY, \op{c}}(\bt, z)(\alpha) := \alpha + \sum_{d \in \op{Eff}} \sum_{k \geq 0} \frac{1}{k!}
I_* \circ{ev_2}^{\op{c}}_* \left(\frac{{ev_1}_{\op{c}}^*(\alpha)}{-z - \psi} \prod_{j=3}^{k+2} ev_j^*(\bt) \cap \left[\sMbar_{0, k+2}(\cY, d)\right]^{vir} \right)
\end{equation}
for $\alpha $ in   
 $H^*_{\op{CR}, \op{c}}(\cY)$.  

There is a completely analogous result to Proposition~\ref{p:flat} in the non-proper case.  
\begin{prop}\label{p:cvsreg} Let $\cY$ be a non-proper space. %
The quantum connection $\nabla^\square$ is flat, with fundamental solution $L^\square(\bt, z) z^{- \op{Gr}}z^{\rho(\cY)}$.  More precisely,
\begin{equation}\label{e:TRR1} \nabla^\cY_i \left(L^\cY(\bt, z) z^{- \op{Gr}}z^{\rho(\cY)} \alpha\right) = \nabla^\cY_z \left(L^\cY(\bt, z) z^{- \op{Gr}} z^{\rho(\cY)}\alpha\right) = 0 \end{equation}  
for  $\alpha \in H^*_{\op{CR}}(\cY)$ and
\begin{equation}\label{e:TRR2} \nabla^{\cY, \op{c}}_i \left(L^{\cY, \op{c}}(\bt, z) z^{- \op{Gr}}z^{\rho(\cY)} \beta\right) = \nabla^{\cY, \op{c}}_z \left(L^{\cY, \op{c}}(\bt, z) z^{- \op{Gr}} z^{\rho(\cY)}\beta\right) = 0 \end{equation}
for   $\beta \in H^*_{\op{CR}, \op{c}}(\cY)$.  
Furthermore the pairing $S^\cY$ satisfies 
\begin{equation} \partial_i S^\cY ( u, v) = S^\cY(\nabla^\cY_i u, v) + S^\cY (u, \nabla^{\cY, \op{c}}_i v).\end{equation}
In other words 
$\nabla^{\cY}$ and $\nabla^{\cY, \op{c}}$ 
are dual with respect to $S^\cY$.
Finally, for $\alpha \in H^*_{\op{CR}}(\cY)$ and $\beta \in H^*_{\op{CR}, \op{c}}(\cY)$, 
\begin{equation} \label{e:Lpair}
\langle L^\cY(\bt, - z) \alpha, L^{\cY, \op{c}}(\bt, z) \beta \rangle^\cY = \langle \alpha, \beta\rangle^\cY.\end{equation}
\end{prop}
\begin{proof}
The proof of these statements is almost identical to the case of a proper target.  The only difference is the precise statement of the topological recursion relation for a non-compact target, which is used to prove \eqref{e:TRR1} and \eqref{e:TRR2}.  In this context the statement is that for $\alpha \in H^*_{\op{CR}, \op{c}}(\cY)$, $\beta, \gamma \in H^*_{\op{CR}}(\cY)$, and $b_1, b_2, b_3 \geq 0$,
\[
\br{\br{ \alpha \psi^{b_1+1}, \beta \psi^{b_2}, \gamma \psi^{b_3} }}^\cY = \sum_{i \in I}
\br{\br{ \alpha \psi^{b_1}, T_i }}^\cY \br{\br{T^i, \beta \psi^{b_2}, \gamma \psi^{b_3} }}^\cY\]
where recall that $\{T_i\}_{i \in I}$ and $\{T^i\}_{i \in I}$ are bases for $H^*_{\op{CR}}(\cY)$ and $H^*_{\op{CR}, \op{c}}(\cY)$ respectively, so both factors in the right hand sum are well-defined.  Similarly, if $\alpha, \beta \in H^*_{\op{CR}}(\cY)$ and $\gamma \in H^*_{\op{CR}, \op{c}}(\cY)$ we have
\[
\br{\br{ \alpha \psi^{b_1+1}, \beta \psi^{b_2}, \gamma \psi^{b_3} }}^\cY = \sum_{i \in I}
\br{\br{ \alpha \psi^{b_1}, T^i }}^\cY \br{\br{T_i, \beta \psi^{b_2}, \gamma \psi^{b_3} }}^\cY.\]
The proof of these statements is identical to the proof in the case of a proper target, after using the following version of the K\"unneth formula
\[ H^*_{\op{CR}}(\cY) \otimes H^*_{\op{CR}, \op{c}}(\cY) \cong H^*_{CR, c-vert}(\cY \times \cY),\]
where the right-hand side denotes cohomology with compact vertical support (i.e. in the second factor).  Under this isomorphism, the class of the diagonal $[\Delta]$ in 
$H^*_{CR, c-vert}(\cY \times \cY)$ is given by $\sum_{i \in I} T_i \otimes T^i$.
\end{proof}

\begin{rema}
The notion above of the compactly supported quantum connection and solution were described previously in Section 2.5 of \cite{Iri2}.  We expect $\nabla^{\cY}$ and $\nabla^{\cY, \op{c}}$ to be related to the conjectures of \cite{BH}.
\end{rema}

\subsection{Narrow quantum $D$-module}\label{s:nar}

We cannot define a (non-equivariant) quantum $D$-module in the case that $\cY$ is non-proper due to the fact that $H^*_{\op{CR}}(\cY)$ does not have a well-defined pairing.  In this section we show that there is a well-defined \emph{narrow} quantum $D$-module for $\cY$, defined in terms of the narrow cohomology of \S~\ref{s:1}.  We will see in \S~\ref{s:narqsd} the geometric significance of this construction.


\begin{prop}\label{p:phi}
The map $\phi: H^*_{\op{CR}, \op{c}}(\cY) \to H^*_{\op{CR}}(\cY)$ commutes with the quantum product, quantum connection, and the fundamental solution: For $u \in H^*_{\op{CR}, \op{c}}(\cY) \otimes P^\cY$ and $T_i \in H^*_{\op{CR}}(\cY)$
\begin{align}
T_i \bullet_{\bt}^\cY \phi(u) = \phi(T_i \bullet_{\bt}^\cY u); \label{e:c1}\\
\nabla_i^{\cY} \phi(u) = \phi \left( \nabla_i^{\cY, \op{c}} u\right); \label{e:c2}\\
L^\cY(\bt, z) \phi(u) = \phi \left( L^{\cY, \op{c}}(\bt, z) u \right). \label{e:c3}
\end{align}
\end{prop}
\begin{proof}
To prove \eqref{e:c1}, it suffices to show that for $T^j, \beta \in H^*_{\op{CR}, \op{c}}(\cY)$
\[\langle T_i \bullet_{\bt}^\cY \phi(\beta), T^j\rangle^\cY = \langle \phi(T_i \bullet_{\bt}^\cY \beta), T^j\rangle^\cY.\]
The left hand side is equal to 
$\br{\br{ T_i, \phi(\beta), T^j }}^\cY$.  By  the fact that $\phi$ commutes with pullback and \eqref{e:phicup}, we note that 
\[ ev_2^*(\phi(\beta)) \cup {ev_3}_{\op{c}}^*(T^j) = \phi( {ev_2}_{\op{c}}^*(\beta)) \cup {ev_3}_{\op{c}}^*(T^j) = {ev_2}_{\op{c}}^*(\beta) \cup {ev_3}_{\op{c}}^*(T^j).\]
This implies that 
  \[\br{\br{ T_i, \phi(\beta), T^j }}^\cY = \br{\br{ T_i, \beta, T^j }}^\cY.\]
 Again by \eqref{e:phicup}, 
 \[\langle \phi(T_i \bullet_{\bt}^\cY \beta), T^j\rangle^\cY = \langle \phi(T^j), T_i \bullet_{\bt}^\cY \beta \rangle^\cY,\]
 but by an identical argument as above, this is given by 
 \[  \br{\br{ T_i, \beta, \phi(T^j) }}^\cY =  \br{\br{ T_i, \beta, T^j }}^\cY.\]
 Thus the two are equal.
 Formula \eqref{e:c2} follows immediately from \eqref{e:c1} and \eqref{e:c3} uses a similar argument.
 \end{proof}
Define a $z$-sesquilinear pairing $S^{\cY, \op{nar}}$ on $H^*_{\op{CR}, \op{nar}}(\cY) \otimes P^\cY[z, z^{-1}]$ by 
\[ S^{\cY, \op{nar}}( u(z), v(z)) := (2 \pi \sqrt{-1} z)^{\dim(\cY)} \langle u(-z), v(z) \rangle^{\cY, \op{nar}}.\]

\begin{coro}\label{c:qsdnar}
For any $\bt \in H^*_{\op{CR}}(\cY)$, the narrow state space is closed under the quantum product $\bullet_{\bt}^\cY$. 
The quantum connection $\nabla^\cY$ and solution $L^\cY( \bt, z)$ preserve $ H^*_{\op{CR}, \op{nar}}(\cY)$. 
The pairing $S^{\cY, \op{nar}}$ is flat with respect to $\nabla^\cY$, i.e.
\[ \partial_i S^{\cY, \op{nar}} ( u, v) = S^{\cY, \op{nar}}(\nabla^\cY_i u, v) + S^\cY (u, \nabla^\cY_i v).\]
Finally, for $\alpha, \beta \in H^*_{\op{CR}, \op{nar}}(\cY)$, 
\[
\langle L^\cY(\bt, - z) \alpha, L^\cY(\bt, z) \beta \rangle^{\cY, \op{nar}} = \langle \alpha, \beta\rangle^{\cY, \op{nar}}.\]
\end{coro}

\begin{proof}
The first two claims are immediate from \eqref{e:c1}, \eqref{e:c2},  \eqref{e:c3}, and the definition of $H^*_{\op{CR}, \op{nar}}(\cY)$ as the image of $\phi$.
The last two claims follow from the same equations together with Proposition~\ref{p:cvsreg}.
\end{proof}

With this we can define
\begin{defi}
The \emph{narrow} quantum $D$-module of $\cY$ is defined to be 
\[ QDM_{\op{nar}}(\cY) := (H^*_{\op{CR}, \op{nar}}(\cY) \otimes P^\cY[z, z^{-1}], \nabla^\cY, S^{\cY, \op{nar}})\]
Note that the coefficients ring $P^\cY$ is \emph{not} restricted to just the dual coordinates of the narrow cohomology.  
\end{defi}

\subsection{Integral structure}
Identical considerations to \S~\ref{ss:fs} in the non-proper case allow one to define dual integral lattices in 
$\op{ker}( \nabla^\cY)$ and $\op{ker}( \nabla^{\cY, c})$, compatible with the Dubrovin connection.
\begin{assumption}\label{as:nc}
The compactly supported Chern character \[\ch^{\op{c}}: D_{\op{c}}(\cY) \to H^*_{\op{CR}, \op{c}}(\cY)\]
is given as Definition~\ref{d:cch} of the appendix.
Assume that $H^*_{\op{CR}, \op{c}}(\cY)$ is spanned by the image of $\ch^{\op{c}}$ and that $H^*_{\op{CR}}(\cY)$ is spanned by the image of 
\[\ch: D(\cY) \to H^*_{\op{CR}}(\cY).\]
\end{assumption}
Importantly, Assumption~\ref{as:nc} holds  if $\cY$ is the total space of a vector bundle $\cE$ on $\cX$ and $H^*_{\op{CR}}(\cX)$ is spanned by the image of $\ch$.
Let $F$ be an object in $D_{\op{c}}(\cY)$, assume $F$ can be represented by a complex $F^\bullet$ which is exact outside a proper substack $\cX$. 
Define $\bs^{\cY, \op{c}}(F)(\bt, z)$ to be 
\[ \frac{1}{(2 \pi \sqrt{-1})^{\dim(\cY)}} L^{\cY, \op{c}}(\bt, z) z^{-Gr} z^{\rho(\cY)} \hat \Gamma_{\cY} \cup_{I\cY}\left( (2 \pi \sqrt{-1})^{\deg_0/2} I^*(\ch^{\op{c}}(F))\right),\]
where $\ch^{\op{c}}(F)  = \ch^{\op{c}}(F^\bullet)$ is given by Definition~\ref{d:cch} and  $\cup_{I\cY}$ denotes the ordinary cup product on $H^*(I\cY)$.
Similarly, for $F \in D(\cY)$, define $\bs^{\cY}(F)(\bt, z)$ to be
\[ \frac{1}{(2 \pi \sqrt{-1})^{\dim(\cY)}} L^{\cY}(\bt, z) z^{-Gr} z^{\rho(\cY)} \hat \Gamma_{\cY} \cup_{I\cY} \left( (2 \pi \sqrt{-1})^{\deg_0/2} I^*(\ch(F))\right).\]
We obtain lattices
\[\{\bs^{\cY, \op{c}}(F)(\bt, z) | F \in D_{\op{c}}(\cY)\} \text{ and } \{\bs^{\cY}(F)(\bt, z) | F \in D(\cY)\},\]
which are dual
with respect to the pairing $S^\cY$.  Proposition~\ref{p:pairing} holds in this context by the same argument.  Namely, 
\[ S^\cY (\bs^\cY(F)(\bt, z), \bs^{\cY, \op{c}}(F ')(\bt, z)) = e^{ \pi \sqrt{-1} \dim(\cY)}\chi(F ', F).\] See Section 2.5 of \cite{Iri2} for a similar description.

By 
Proposition~\ref{p:cnar} the orbifold Chern character map $\ch: D(\cY) \to H^*_{\op{CR}}(\cY)$ maps $D_{\op{c}}(\cY)$ to $H^*_{\op{CR}, \op{nar}}(\cY)$.  
Therefore,
given an object $F$ in $D_{\op{c}}(\cY)$,  define $\bs^{\cY, \op{nar}}(F)(\bt, z)$ as
\[ \frac{1}{(2 \pi \sqrt{-1})^{\dim(\cY)}} L^\cY(\bt, z) z^{-Gr} z^{\rho(\cY)} \hat \Gamma_{\cY} \cup_{I\cY} \left( (2 \pi \sqrt{-1})^{\deg_0/2} I^*(\ch(F))\right).\]  
\begin{defi} Define the integral structure for $QDM_{\op{nar}}(\cY)$ to be
\[ \{ \bs^{\cY, \op{nar}}(F)(\bt, z) | F \in D_{\op{c}}(\cY) \}.\]  This set forms a lattice in $\op{ker}( \nabla^\cY|_{H^*_{\op{CR}, \op{nar}}(\cY) })$ 
\end{defi}

\section{Equivariant Euler twistings}\label{s:twisting}
In many cases, twisted invariants of a vector bundle $\cE \to \cX$ are closely related to Gromov--Witten invariants of both the total space and a corresponding complete intersection.  In this section we recall the connections to each, with a particular focus on the relationship to the \emph{non-equivariant} Gromov--Witten theory of the total space of $\cE^\vee$, which has not been studied as thoroughly as the complete intersection.

Let $\cE \to \cX$ be a convex vector bundle.  Let 
\[j: \cZ \to \cX\] denote a smooth sub-variety of $\cZ$ defined by the vanishing of a regular section of $\cE$. 
Let $T = \CC^*$ act on $\cE$ by scaling in 
the fiber direction, with equivariant parameter $\lambda$.

\subsection{Subvarieties}
Choose $s_k$ such that  $\bs(\cE) = e_\lambda(\cE)$, the equivariant Euler characteristic:
\begin{equation}\label{e:specCI}
s_0 = \ln(\lambda)\text{, } s_k = (-1)^{k-1}(k-1)!/\lambda^k \text{ for } k > 0.
\end{equation}
In this case, the genus-zero $\bs$-twisted invariants with respect to $\cE$ are related to invariants of the local complete intersection subvariety $\cZ$ cut out by a generic section of $\cE$ by the so-called \emph{quantum Lefschetz principle} (\cite{CG, Ts}).  
One way of phrasing this is the following:
\begin{prop}[Proposition 2.4 \cite{Iri3}] \label{p:noneq1} Let $L^{e_\lambda(\cE)}(\bt, z)$ denote the fundamental solution of the equivariant twisted theory of $\cE$ after specializing parameters as in \eqref{e:specCI}.  Then the non-equivariant limit
\[L^{e(\cE)}(\bt, z) := \lim_{\lambda \mapsto 0} L^{e_\lambda(\cE)}(\bt, z)\]
is well defined.  Furthermore,
for $\alpha \in H^*_{\op{CR}}(\cX)$,
\[ j^*\left(L^{e(\cE)}(\bt, z) \alpha \right) = L^\cZ(j^*(\bt), z) j^*(\alpha).\]
\end{prop}

\subsection{The total space}\label{ss:ts}
On the other hand, let $\cY$ denote the total space of $\cE^\vee$.  One can consider the equivariant Gromov--Witten invariants of $\cY$ with respect to the torus action described above.  We will denote the equivariant Gromov--Witten invariants by 
$\langle \alpha_1 \psi^{b_1}, \ldots, \alpha_n \psi^{b_n} \rangle^{\cY_T}$.
These take values in $H^*_T(\op{pt}) = \CC[\lambda]$, the equivariant cohomology of a point.  Let 
 $\pi: \cY \to \cX$ denote the projection.
If we specialize the twisted parameters to
\begin{equation}\label{e:spec2}
s_0 ' = -\ln(-\lambda)\text{, } s_k ' = (k-1)!/\lambda^k \text{ for } k > 0
\end{equation}
 then $\bs '(\cE^\vee) = e_\lambda^{-1}(\cE^\vee)$. 
 In this case by (virtual) Atiyah--Bott localization \cite{GP},  the $e_\lambda^{-1}(\cE^\vee)$-twisted invariants compute the equivariant invariants of $\cY$ after inverting the equivariant parameter.
 Given $\alpha_1, \ldots, \alpha_n \in H^*_{\op{CR}, T}(\cX)$ and $b_1, \ldots, b_n \in \ZZ_{\geq0}$, we have 
\[ \langle \alpha_1 \psi^{b_1}, \ldots, \alpha_n \psi^{b_n} \rangle^{e_\lambda^{-1}(\cE^\vee)}_{h, n, d} := 
\int_{[\sMbar_{h,n}(\cc X, d)]^{vir}}   \frac{ \prod_{i=1}^n ev_i^*( \alpha_i) \cup \psi_i^{b_i}}{ e_\lambda(\mathbb R \pi_* f^*(\cE)^\vee)}.\]
 
 

\begin{prop}\label{p:noneq2}  Let $L^{e_\lambda^{-1}(\cE^\vee)}(\bt, z)$ denote the fundamental solution of the equivariant twisted theory of $\cE^\vee$ after specializing parameters as in \eqref{e:spec2}.  Then the non-equivariant limit
\[L^{e^{-1}(\cE^\vee)}(\bt, z) := \lim_{\lambda \mapsto 0} L^{e_\lambda^{-1}(\cE^\vee)}(\bt, z)\]
is well defined.  Furthermore,
for $\alpha \in H^*_{\op{CR}}(\cX)$,
\[ \pi^*\left(L^{e^{-1}(\cE^\vee)}(\bt, z) \alpha \right) = L^\cY(\pi^*(\bt), z) \pi^*(\alpha).\]
\end{prop}
\begin{proof}
By \cite[Section 4]{GivE} using the virtual localization formula of \cite{GP}, the equivariant Gromov--Witten invariants of $\cY$ may be calculated as twisted invariants of $\cX$:
\begin{equation}\label{e:equivinv}
\langle \alpha_1 \psi^{b_1}, \ldots, \alpha_n \psi^{b_n} \rangle^{e_\lambda^{-1}(\cE^\vee)}_{h, n, d} =
 \langle \alpha_1 \psi^{b_1}, \ldots, \alpha_n \psi^{b_n} \rangle^{\cY_T}_{h, n, d} 
\end{equation}
where $\cY_T$ denotes the $T$-equivariant Gromov--Witten theory of $\cY$.
By the (non-virtual) localization theorem, the $e_\lambda^{-1}(\cE^\vee)$-twisted pairing agrees with the equivariant pairing on $\cY$:
\begin{equation}\label{e:equivpair} \langle \alpha, \beta \rangle^{e_\lambda^{-1}(\cE^\vee)} = \langle \pi^*(\alpha), \pi^*(\beta) \rangle^{\cY_T}.\end{equation}
Thus if we are given a basis $\{T_i\}_{i \in I}$ for $\cX$ and a dual basis $\{T^i\}_{i \in I}$ with respect to the $e_\lambda^{-1}(\cE^\vee)$-twisted pairing, then $\{\pi^*(T_i)\}_{i \in I}$ and $\{\pi^*(T^i)\}_{i \in I}$ will be dual bases with respect to the equivariant pairing on $\cY$.

By comparing term-by-term, we conclude that the fundamental solutions agree:
\begin{equation}\label{e:toteq}\pi^*\left(L^{e_\lambda^{-1}(\cE^\vee)}(\bt, z)\alpha\right) = L^{\cY_T}(\pi^*(\bt), z)\pi^*(\alpha).\end{equation}
Because we assume $\cE$ is convex, by Lemma~\ref{l:propmap} the genus-zero evaluation maps are proper and the non-equivariant limit of the right hand side is well defined.  

If we pull back the left hand side via the map $i: \cX \to \cY$ we recover $L^{e_\lambda^{-1}(\cE^\vee)}(\bt, z)\alpha$, which proves the first statement.  
%
Taking the non-equivariant limit of \eqref{e:toteq} finishes the proof.
\end{proof}
In Proposition~\ref{p:ringiso}, we describe the non-equivariant limit of the $e_\lambda^{-1}(\cE^\vee)$-twisted quantum product more explicitly.  To our knowledge this description is new.  It is similar (or more accurately dual) in spirit to the 
$\bullet^{Z}_\bt$ product of \cite{Pan}. 

\begin{lemma}\label{l:hnot}
Assume that $\cE$ is convex. 
For a stable map $f: (\cC, p_1, \ldots, p_n) \to \cX$ from a genus zero $n$-marked orbi-curve $\cC$, \[H^0\left(\cC, f^*(\cE^\vee)(-p_i)\right) = 0\]
for any choice of $1 \leq i\leq n$.
\end{lemma}
\begin{proof}
Let $r: \cC \to C$ be the map to the underlying coarse curve.  By Remark~\ref{r:conv} $\cE$ is the pullback of a convex vector bundle $E \to X$.
By Theorem 1.4.1 of \cite{AV}, the map $ \cC \to \cX \to X$ factors through a map $|f|: C \to X$.  Therefore $f^*\cE = r^* |f|^* E$.  We observe that
\begin{align*} H^0\left(C, r_*(f^*(\cE^\vee)(-p_i))\right) &= H^0\left(C, r_*(r^*|f|^*(E^\vee)(-p_i))\right) \\
&= H^0\left(C, |f|^*E^\vee \otimes r_*(\cc O_{\cC}(-p_i))\right)\\
&=
H^0\left(C, |f|^*E^\vee\otimes \cc O_C(-p_i)\right),
\end{align*}
where the second equality is the projection formula and the third can be checked using local coordinates.
 On each irreducible component $C_j$ of $C$, $E \cong \sum_{l=1}^r \cc O(k_l)$ with each $k_l \geq 0$.  From this we see that the only global sections of $|f|^*(E^\vee)(-p_i)|_{C_j}$ are constant sections (possibly just the zero section) if $p_i$ is not on $C_j$, and the zero section if $p_i$ does lie on $C_j$.  By an induction argument on the number of components, the only global section of $|f|^*(E^\vee)(-p_i)$ is the zero section.
\end{proof}
Consider the short exact sequence over the universal curve $\widetilde{\cC}$ lying over $\sMbar_{h,n}(\cc X, d)$:
\[ 0 \to f^*(\cE)^\vee(-p_i) \to f^*(\cE)^\vee \to f^*(\cE)^\vee|_{p_i} \to 0.\]
Pushing forward we obtain the distinguished triangle 
\[ \mathbb R \pi_* f^*(\cE)^\vee(-p_i) \to \mathbb R \pi_* f^*(\cE)^\vee \to  ev_i^*(q^*(\cE^\vee))[1],\]
where $q: I \cX \to \cX$ is the natural map.

From this and the previous lemma we can rewrite the twisted invariant $\langle \alpha_1 \psi^{b_1}, \ldots, \alpha_n \psi^{b_n} \rangle^{e_\lambda^{-1}(\cE^\vee)}_{0, n, d}$ as 
\begin{equation}\label{e:newex}
\int_{[\sMbar_{0,n}(\cc X, d)]^{vir}}   \prod_{j=1}^n \left(ev_j^*( \alpha_j) \cup \psi_j^{b_j} \right) \cup\frac{  e_\lambda(\mathbb R^1 \pi_* f^*(\cE)^\vee(-p_i))}{ ev_i^*(q^*(e_\lambda(\cE^\vee)))}.
\end{equation}
The above expression motivates the following definition.
\begin{defi}
Fix $i$ between $1$ and $n$.  Given $\alpha_1, \ldots, \alpha_n \in H^*_{\op{CR}, T}(\cX)$ and $b_1, \ldots, b_n \in \ZZ_{\geq0}$, define $ \langle \alpha_1 \psi^{b_1}, \ldots, \widetilde{\alpha_i \psi^{b_i}}, \ldots, \alpha_n \psi^{b_n} \rangle^{e^{-1}(\cE^\vee)}_{0, n, d} $ to be the integral
\begin{align*} 
\int_{[\sMbar_{0,n}(\cc X, d)]^{vir}}  \prod_{j=1}^n \left(ev_j^*( \alpha_j) \cup \psi_j^{b_j} \right) \cup  e(\mathbb R^1 \pi_* f^*(\cE)^\vee(-p_i)).\end{align*} 
\end{defi}
Note that $\mathbb R^1 \pi_* f^*(\cE)^\vee(-p_i)$ may be represented by a vector bundle by Lemma~\ref{l:hnot}.
We now define a new quantum product on $\cX$:
\begin{defi}\label{d:mqp}
Let $\{ \overline T_i\}$ be a basis for $H^*_{\op{CR}}(\cX)$ and let $\{ \overline T^i\}$ denote the dual basis.  For $\alpha$ and $\beta \in H^*_{\op{CR}}(\cX)$, define
\[ \alpha \bullet^{\cY \to \cX}_\bt \beta := \br{ \br{ \alpha, \beta, \widetilde{\overline T_i} }}^{e^{-1}(\cE^\vee)} \overline T^i.\]
\end{defi}

\begin{prop}\label{p:ringiso}
The pullback $$\pi^*: H^*_{\op{CR}}(\cX) \otimes P^{\cX} \to H^*_{\op{CR}}(\cY) \otimes P^{\cY}$$ is a ring isomorphism from the $\bullet^{\cY \to \cX}_\bt$-product on $\cX$ to the quantum product $\bullet_\bt^\cY$ on $\cY$.
\end{prop}
\begin{proof}
We first show that in the non-equivariant limit, the product $\bullet^{e_\lambda^{-1}(\cE^\vee)}_\bt$ specializes to 
 $\bullet^{\cY \to \cX}_\bt$.
 
The bases $\{\overline  T_i\}$ and  $\{\overline  T^i\}$ give dual bases with respect to the equivariant pairing on $H^*_{\op{CR}, T}(\cX)$ via the inclusion $H^*_{\op{CR}}(\cX)  \subset H^*_{\op{CR}, T}(\cX) \cong H^*_{\op{CR}}(\cX) \otimes \CC[\lambda]$. Define $ T_i := e_\lambda(\cE^\vee) \cup  \overline T_i$ and $  {T^i} := \overline  T^i$.  Note that with respect to the $e_\lambda^{-1}(\cE^\vee)$-twisted pairing, $\{ T_i \}$ and $\{ {T^i} \}$ are dual bases.  Therefore, for $\alpha, \beta \in H^*_{\op{CR}}(\cX)$, 
\begin{align}\label{e:equivprod}
\alpha \bullet_\bt^{e_\lambda^{-1}(\cE^\vee)} \beta &= 
\sum_{i \in I}\br{ \br{ \alpha, \beta, T_i  }}^{e^{-1}_\lambda(\cE^\vee)} {T^i } \\ \nonumber
& = \sum_{i \in I}\br{ \br{ \alpha, \beta, e_\lambda(\cE^\vee) \cup  \overline T_i  }}^{e_\lambda^{-1}(\cE^\vee)} {\overline T^i }.
\end{align}

By \eqref{e:newex} and Remark~\ref{remaCR}, the factor of $e_\lambda(\cE^\vee)$ in the third insertion cancels with part of the twisted virtual class and this expression becomes
\begin{align*} \sum_{i \in I}  \sum_{d \in \op{Eff}} \sum_{k \geq 0} \frac{ T^i}{k!} \int_{[\sMbar_{0,k+3}(\cc X, d)]^{vir}} &ev_1^*(\alpha) \cup ev_2^*(\beta) \cup ev_3^*(T_i)\cup \\ &\prod_{j=4}^{k+3} ev_j^*(\bt) \cup e_\lambda (\mathbb R^1 \pi_* f^*(\cE)^\vee(-p_3)).\end{align*}
In the nonequivariant limit $\lambda \mapsto 0$, this is exactly $\alpha \bullet^{\cY \to \cX}_\bt \beta$.

On the other hand, by~\eqref{e:equivinv} and~\eqref{e:equivpair}, 
if we pull back~\eqref{e:equivprod} by $\pi$ we obtain
\[
\sum_{i \in I} \br{ \br{ \pi^*(\alpha), \pi^*(\beta), \pi^*( T_i)  }}^{\cY_T} \pi^*(T^i),\]
which may be rewritten as
\[
\sum_{i \in I} \sum_{d \in \op{Eff}} \sum_{k \geq 0} \frac{1}{k!} I_* \circ{ev_3}_* \left( ev_1^*\pi^*(\alpha) \cup ev_2^*\pi^*(\beta)\cup \prod_{j=4}^{k+3} ev_j^*\pi^*(\bt) \cap \left[ \sMbar_{0,k+3}(\cY_T, d)\right]^{vir} \right),
\] where $\left[ \sMbar_{0,k+3}(\cY_T, d)\right]^{vir}$ is the equivariant virtual fundamental class.
In the non-equivariant limit we recover $\pi^*(\alpha) \bullet^{\cY}_{\pi^*(\bt)} \pi^*(\beta)$ by \eqref{e:noncompprod}.

Combining the above results we conclude
\begin{align*} \pi^*\left(\alpha \bullet^{\cY \to \cX}_\bt \beta\right) &=  \pi^*\left(\lim_{\lambda \mapsto 0} \alpha \bullet_\bt^{e_\lambda^{-1}(\cE^\vee)} \beta\right) \\
&= \lim_{\lambda \mapsto 0} \pi^*\left(\alpha \bullet_\bt^{e_\lambda^{-1}(\cE^\vee)} \beta\right) \\
&=  \lim_{\lambda \mapsto 0} \left( \pi^*(\alpha) \bullet^{\cY_T}_{\pi^*(\bt)} \pi^*(\beta) \right) \\
&=\pi^*(\alpha) \bullet^{\cY}_{\pi^*(\bt)} \pi^*(\beta).\end{align*}
%
\end{proof}

\section{Quantum Serre duality}\label{s:qsd}

In this section we use the definition of the compactly supported quantum connection and the narrow quantum $D$-module to reframe quantum Serre duality in two new ways.  First we relate the compactly supported quantum connection of $\cY$ to the quantum connection of $\cZ$.  Second, we show there is an \emph{isomorphism} between the narrow quantum $D$-module of $\cY$ and the ambient quantum $D$-module of $\cZ$.  In both cases we show these correspondences to be compatible with the integral structures.

In all of this section we assume: 
\begin{itemize}
\item The vector bundle $\cE \to \cX$ is convex;
\item Assumption~\ref{a:ass2};
\item Assumption~\ref{as2};
\item The stack $\cX$ has the resolution property.
\end{itemize}

\subsection{Compactly supported quantum Serre duality}

In \cite{IMM}, it is observed (Remark~3.17) that the $e(\cE)$-twisted quantum $D$-module can be viewed as the quantum $D$-module with compact support of the total space $\cE^\vee$.  We make this observation precise by relating the Euler twisted fundamental solution $L^{e(\cE)}(\bt, z)$ with the compactly supported fundamental solution $L^{\cY, \op{c}}(\bt, z)$.  This 
then allows us to directly relate the compactly supported fundamental solution of $\cY$ with the ambient fundamental solution $L^\cZ(\bt, z)$, obtaining a new perspective on quantum Serre duality.  In Remark~\ref{r:IMMr} we explain that the results of this section should be viewed as adjoint to a similar theorem in \cite{IMM}.  The majority of the techniques of this section appeared already in \cite{IMM}, it is mainly the perspective which is new.  This particular formulation of quantum Serre duality, described in Theorem~\ref{t:dcom}, is convenient for then proving the more refined statement in \S~\ref{s:narqsd}.

Recall the definition of $P^\cX$ given in Notation~\ref{n:P}.
Consider the map $\hat f^\lambda(\bt): H^*_{\op{CR}, T}(\cX) \to H^*_{\op{CR}, T}(\cX) \otimes P^\cX[[\lambda]]$ given by
\begin{equation}\label{e:covf} \hat f^\lambda(\bt) := \sum_{i \in I} \br{\br{ e_\lambda(\cE^\vee), T_i}}^{e_\lambda(\cE^\vee)} T^i\end{equation}
where $\{T_i\}$ and $\{T^i\}$ as dual bases with respect to the $e_\lambda^{-1}(\cE^\vee)$-pairing. 

\begin{prop}\label{p:noncov}
The map $\hat f^\lambda(\bt)/e_\lambda(\cE^\vee)$ has a well defined non-equivariant limit $\hat f^\cX(\bt)$ given by
\[ 
\sum_{i \in I} \br{\br{ \widetilde 1 , \overline T_i}}^{e^{-1}(\cE^\vee)} \overline T^i,\]
where $\{\overline T_i\}$ and $\{\overline T^i\}$ as dual bases with respect to the usual pairing on $H^*_{\op{CR}}(\cX)$.
\end{prop}
\begin{proof}
Due to the difference in the usual pairing on $H^*_{\op{CR}, T}(\cX)$ and the $e_\lambda^{-1}(\cE^\vee)$-twisted pairing, $\hat f^\lambda(\bt)/e_\lambda(\cE^\vee)$ can be expressed as 
$$\sum_{i \in I} \br{\br{ e_\lambda(\cE^\vee)  , \overline T_i}}^{e_\lambda^{-1}(\cE^\vee)} \overline T^i.$$
The claim then follows by the same argument as in Proposition~\ref{p:ringiso}. 
\end{proof}

\begin{prop}\label{p:Lcomc}
Let $\bar f^\cX(\bt) = \hat f^\cX(\bt) - \pi \sqrt{-1} c_1(\cE)$.
The isomorphism 
\[i^{\op{c}}_*: H^*_{\op{CR}}(\cX) \to H^*_{\op{CR}, \op{c}}(\cY)\]
identifies the connections $(\bar f^\cX \circ i^*)^*\left(\nabla^{e(\cE)}\right)$ and $\nabla^{\cY, \op{c}}$.
Furthermore,
\begin{equation}\label{e:Lcs}L^{\cY, \op{c}}(\bt, z) i^{\op{c}}_*(\beta) = i^{\op{c}}_* \left(  L^{e(\cE)}(\bar f^\cX( i^*(\bt)), z) e^{- \pi \sqrt{-1} c_1(\cE)/z} \beta \right)
\end{equation}
for all $\beta \in H^*_{\op{CR}}(\cX)$.
\end{prop}
\begin{proof}
The first claim follows from the second.  
By Theorem~\ref{t:QSDlag}, the symplectic transformation 
\[\Delta^\diamond := 
 e^{ \pi \sqrt{-1} c_1(\cE)/z}/e_\lambda(\cE^\vee)\]
maps $\sL^{e_{\lambda}^{-1}(\cE^\vee)}$ to $\sL^{e_\lambda(\cE)}$.

By~\eqref{e:genfam} and~\eqref{e:Jgens}, $\Delta^\diamond z \partial_{-e_\lambda(\cE^\vee)} J^{e_{\lambda}^{-1}(\cE^\vee)}(\bt, -z)$ is a $\CC[z]$-linear combination of derivatives of $J^{e_\lambda(\cE)}(\hat \bt, -z)$ at some point $\hat \bt$.  
Observe that 
\begin{align*}
&\Delta^\diamond z \partial_{-e(\cE^\vee)} J^{e_{\lambda}^{-1}(\cE^\vee)}(\bt, -z) \\
=& \frac{1}{e_\lambda(\cE^\vee)} \left(1 + \pi \sqrt{-1} c_1(\cE)/z \right)\left( -e_\lambda(\cE^\vee) z + \sum_{i \in I} \br{\br{e_\lambda(\cE^\vee), T_i}}^{e_{\lambda}^{-1}(\cE^\vee)}T^i \right) + \cc O(1/z) \\ 
=& -z + \frac{1}{e_\lambda(\cE^\vee)} \left( \sum_{i \in I} \br{\br{e_\lambda(\cE^\vee), T_i}}^{e_{\lambda}^{-1}(\cE^\vee)}T^i \right) - \pi \sqrt{-1} c_1(\cE) + \cc O(1/z).
\end{align*}
From this we see that $\Delta^\diamond z \partial_{-e(\cE^\vee)} J^{e_{\lambda}^{-1}(\cE^\vee)}(\bt, -z)$ is equal to $J^{e_\lambda(\cE)}(\bar f^\lambda(\bt), -z)$ where 
\begin{align*}\bar f^\lambda(\bt) &=  \frac{1
}{e_\lambda(\cE^\vee)} \left( \sum_{i \in I} \br{\br{e_\lambda(\cE^\vee), T_i}}^{e_{\lambda}^{-1}(\cE^\vee)}T^i \right) - \pi \sqrt{-1} c_1(\cE)\\
&= \hat f^\lambda(\bt)/e_\lambda(\cE^\vee) - \pi \sqrt{-1} c_1(\cE).
\end{align*}
This implies that \[\Delta^\diamond (T_{J^{e_{\lambda}^{-1}(\cE^\vee)}(\bt, -z)} \sL^{e_{\lambda}^{-1}(\cE^\vee)}) = T_{J^{e_\lambda(\cE)}(\bar f^\lambda(\bt), -z)} \sL^{e_\lambda(\cE)}
\]
and so by \eqref{e:Jgens2}, $\Delta_+( \partial_{\alpha}J^{e_{\lambda}^{-1}(\cE^\vee)}(\bt, -z))$ is a $\CC[z]$-linear combination of derivatives of $J^{e_\lambda(\cE)}(\bar \bt, -z)$ evaluated at $\bar \bt = \bar f^\lambda(\bt)$. Comparing $z^0$-coefficients, we see
\begin{equation*}\frac{e^{- \pi \sqrt{-1} c_1(\cE)/z}}{ e_\lambda(\cE^\vee)} \partial_{e_\lambda(\cE^\vee)\cup\alpha} J^{e_{\lambda}^{-1}(\cE^\vee)}(\bt, z) =
\partial_{\alpha} J^{e_\lambda(\cE)}(\bar \bt, z)|_{\hat \bt = \hat f(\bt)},\end{equation*}  where we have replaced $-z$ by $z$.  By \eqref{e:LJ}, this equation can be written as
\begin{equation*}
\frac{e^{- \pi \sqrt{-1} c_1(\cE)/z}}{ e_\lambda(\cE^\vee)} \left(L^{e_\lambda^{-1}(\cE^\vee)}(\bt, z)^{-1} e_\lambda(\cE^\vee)\cup \alpha \right)=
L^{e_\lambda(\cE)}(\bar \bt, z)^{-1} \alpha \end{equation*} or, equivalently, 
\begin{equation}\label{e:deltaJtwisted2}
L^{e_\lambda^{-1}(\cE^\vee)}(\bt, z) (e_\lambda(\cE^\vee)\cup \alpha) =
 e_\lambda(\cE^\vee) \cup L^{e_\lambda(\cE)}(\bar \bt, z)e^{- \pi \sqrt{-1} c_1(\cE)/z} \alpha. \end{equation}

By Proposition~\ref{p:noncov}, the non-equivariant limit of $\hat f^\lambda(\bt)/e_\lambda(\cE^\vee)$ exists.
Then by Proposition~\ref{p:noneq1} the right side therefore has a non-equivariant limit  for $\alpha \in H^*_{\op{CR}, T}(\cX) \subset H^*_{\op{CR}, T}(\cX) \otimes_{R_T} S_T$.

To finish the proof, let $\alpha, \beta \in H^*_{\op{CR}}(\cX)$ and consider the following:
\begin{align*}
\br{L^\cY(\bt, -z) \pi^*(\alpha), L^{\cY, \op{c}}(\bt, z) i^{\op{c}}_* \beta}^\cY & = \br{ \pi^*(\alpha),  i^{\op{c}}_* \beta}^\cY  \\
& = \br{\alpha,\beta}^\cX 
\end{align*}
where the first equality is \eqref{e:Lpair} and the second is the projection formula \cite{BT}. Note that this equation completely determines $L^{\cY, \op{c}}(\bt, z)$ in terms of $L^{\cY}(\bt, z)$.
On the other hand, we have
\begin{align*}
&\br{\alpha,\beta}^\cX \\
= &  \lim_{\lambda \mapsto 0} \br{\alpha, \beta}^{\cX} \\
= &  \lim_{\lambda \mapsto 0} \br{ \alpha, e_\lambda(\cE^\vee) \cup \beta}^{e_{\lambda}^{-1}(\cE^\vee)} \\
= &  \lim_{\lambda \mapsto 0} \br{ L^{e_\lambda^{-1}(\cE^\vee)}(i^*(\bt), -z)  \alpha, L^{e_\lambda^{-1}(\cE^\vee)}(i^*(\bt), z) (e_\lambda(\cE^\vee) \cup \beta)}^{e_{\lambda}^{-1}(\cE^\vee)} \\
= &  \lim_{\lambda \mapsto 0} \br{ L^{e_\lambda^{-1}(\cE^\vee)}(i^*(\bt), -z)  \alpha, e_\lambda(\cE^\vee) \cup L^{e_\lambda(\cE)}(\bar f^{\lambda}(i^*(\bt)), z)e^{- \pi \sqrt{-1} c_1(\cE)/z} \beta}^{e_{\lambda}^{-1}(\cE^\vee)} \\
= &  \lim_{\lambda \mapsto 0} \br{ L^{e_\lambda^{-1}(\cE^\vee)}(i^*(\bt), -z)  \alpha, L^{e_\lambda(\cE)}(\bar f^{\lambda}(i^*(\bt)), z)e^{- \pi \sqrt{-1} c_1(\cE)/z} \beta}^{\cX} \\
= &   \br{ L^{e^{-1}(\cE^\vee)}(i^*(\bt), -z)  \alpha, L^{e(\cE)}(\bar f^\cX(i^*(\bt)), z)e^{- \pi \sqrt{-1} c_1(\cE)/z} \beta}^{\cX} \\
= &   \br{ \pi^*\left( L^{e^{-1}(\cE^\vee)}(i^*(\bt), -z)  \alpha \right), i^{\op{c}}_* \left( L^{e(\cE)}(\bar f^\cX(i^*(\bt)), z)e^{- \pi \sqrt{-1} c_1(\cE)/z} \beta \right)}^{\cY} \\
= &   \br{ L^{\cY}(\bt, -z)  \pi^*(\alpha), i^{\op{c}}_* \left( L^{e(\cE)}(\bar f^\cX(i^*(\bt)), z)e^{- \pi \sqrt{-1} c_1(\cE)/z} \beta \right)}^{\cY}
\end{align*}
The third equality is \eqref{e:invadj}, the fourth is \eqref{e:deltaJtwisted2}, the fifth is from the difference in the twisted and untwisted pairing, the seventh is again the projection formula and the last is Proposition~\ref{p:noneq2}.

Combining the above two chains of equalities yields \eqref{e:Lcs}.

\end{proof}

\begin{defi}
Define \[\Delta^{\op{c}}_+:= j^* \circ \pi^{\op{c}}_*: H^*_{\op{CR}, \op{c}}(\cY) \to H^*_{\op{CR}, \op{amb}}(\cZ).\]  Define 
\[\bar \Delta^{\op{c}}_+ := (2 \pi \sqrt{-1} z)^{\op{rk}(\cE)} \Delta^{\op{c}}_+.\]
\end{defi}

We will show that $\bar \Delta_+^{\op{c}}$  is compatible with the quantum connections, integral structures and the functor 
\[j^* \circ \pi_*: D(\cY)_{\cX} \to D(\cZ).\]
\begin{lemma}\label{l:derind}  Assume that $\cX$ has the resolution property.
Consider the functor $j^* \circ \pi_*: D(\cY)_{\cX} \to D(\cZ)$.  The induced map on cohomology from $H^*_{\op{CR}, \op{c}}(\cY)$ to $H^*_{\op{CR}, \op{amb}}(\cZ)$  is given by $\Delta_+^{\op{c}} (- \cup\op{Td}(  \pi^*\cE^\vee ))$,
i.e.
\begin{equation}\label{e:chcomm} \Delta_+^{\op{c}} \left( \ch^{\op{c}} (F) \cup\op{Td}(  \pi^*\cE^\vee ) \right)= \ch \circ j^* \circ \pi_*(F)\end{equation}
for all $F \in D(\cY)_{\cX}$, where  $\ch^{\op{c}}$ is defined in Definition~\ref{d:cch}.
\end{lemma} 
\begin{proof}
By, Lemmas 4.6 and 4.8 of \cite{BFK2}, $D(\cY)_{\cX}$ is strongly generated by $i_*\left(D(\cX)\right)$.
Thus any element of $D(\cY)_{\cX}$ may be expressed, via a finite sequence of extensions, in terms of elements in $i_*\left(D(\cX)\right)$.
It therefore suffices to check the statement when $F = i_*(G)$ for some $G \in D(\cX)$.  By orbifold Grothendieck--Riemann--Roch \cite{Ts, To}, 
$ \ch^{\op{c}}(i_*(G)) \cup\op{Td}(  \pi^*\cE^\vee ) = i_*(\ch(G)).$  Then,
\begin{align*}
\Delta_+^{\op{c}}( \ch^{\op{c}}(i_*(G)) \cup\op{Td}(  \pi^*\cE^\vee ) ) & = \Delta_+^{\op{c}}(i_* (\ch(G))) \\
& = j^*(\ch(G)) \\
& = j^*(\ch(\pi_* \circ i_* G)) \\
&= \ch(j^*\circ \pi_* \circ i_* G)).
\end{align*}
\end{proof}

\begin{theo}\label{t:dcom}
$\bar \Delta^{\op{c}}_+$ maps $(\cY, c)$-flat sections to $\cZ$-flat sections.  In particular,
\begin{equation}\label{e:comc}
\bar \Delta^{\op{c}}_+ \circ L^{\cY, \op{c}}(\bt, z) (\beta) = L^{\cZ, \op{amb}}(j^* \circ \bar f^\cX\circ i^*(\bt), z) \circ \bar \Delta^{\op{c}}_+ \circ e^{- \pi \sqrt{-1} c_1(\pi^*\cE)/z} \beta.
\end{equation}
Furthermore it is compatible with the integral structure and the functor $j^* \circ \pi_*$, i.e., the following diagram commutes;
\begin{equation}\label{e:squarec}
\begin{tikzcd}
D(\cY)_{\cX} \ar[r, " j^* \circ \pi_*"] \ar[d, "s^{\cY, \op{c}}
"] &  j^*(D(\cX)) \ar[d, "s^{\cZ, \op{amb}}
"]\\
\ker (\nabla^{\cY, \op{c}}) \ar[r, "\bar \Delta_+^{\op{c}}"] & \ker \left((j^* \circ \bar f^\cX \circ i^*)^*(\nabla^{\cZ})\right).
\end{tikzcd}
\end{equation}

\end{theo}
\begin{proof}
The first statement is an immediate consequence of the previous proposition, the fact that $\pi^{\op{c}}_*$ is the inverse of $i^{\op{c}}_*$, and Proposition~\ref{p:noneq1}.  

 To show that  \[ \bar \Delta_+^{\op{c}}\circ s^{\cY, \op{c}}(\bt, z)(F) = s^{\cZ, \op{amb}}(j^* \circ \bar f^\cX\circ i^*(\bt), z) \circ j^*\circ \pi_*(F),\]   note the following facts.  First, recalling the definition of $\hat \Gamma_\cY$ from Definition~\ref{d:Gamma}, we see that 
\begin{align}\label{e:gamma}
\hat \Gamma_\cY &= \pi^*(\hat \Gamma_\cX \hat \Gamma( \cE^\vee)) \\ \nonumber
&= \pi^*\left(\frac{\hat \Gamma_\cX}{\hat \Gamma(\cE)}  \hat \Gamma( \cE^\vee)  \hat \Gamma( \cE)\right)\\ \nonumber
& = \pi^*\left(\frac{\hat \Gamma_\cX}{\hat \Gamma(\cE)} \prod_{j=1}^{\op{rk}(\cE)}\Gamma (1 - \rho_j) \Gamma (1 + \rho_j)\right) \\ \nonumber
& = \pi^*\left(\frac{\hat \Gamma_\cX}{\hat \Gamma(\cE)}  \prod_{j=1}^{\op{rk}(\cE)} \frac{(2 \pi \sqrt{-1}) e^{ \pi \sqrt{-1} \rho_j}(-\rho_j)}{1 - e^{2 \pi \sqrt{-1} \rho_j}}\right)\\ \nonumber
& =  \pi^*\left(e^{ \pi \sqrt{-1} c_1(\cE)} \frac{\hat \Gamma_\cX}{\hat \Gamma(\cE)} (2 \pi \sqrt{-1})^{\deg_0/2} \op{Td}(\cE^\vee)\right) \nonumber
\end{align}
where $\rho_j$ are the Chern roots of $\cE$.
Second, 
\[j^*\left(\frac{\hat \Gamma_\cX}{\hat \Gamma(\cE)}\right) = \hat \Gamma_\cZ.\]
Observing that 
\begin{align*} & \Delta_+^{\op{c}}( z^{-Gr} z^{\rho(\cY)} (2 \pi \sqrt{-1})^{\deg_0/2}(-)) \\ = &\left(\frac{z}{2 \pi \sqrt{-1}}\right)^{\op{rk}(\cE)} z^{-Gr} z^{\rho(\cZ)} (2 \pi \sqrt{-1})^{\deg_0/2}\Delta_+^{\op{c}}(-),\end{align*}
by Lemma~\ref{l:derind} we then have that for all $F \in D(\cY)_{\cX}$,  
\begin{align}\label{e:gamma2} &\left(\frac{z}{2 \pi \sqrt{-1}}\right)^{\op{rk}(\cE)} \Delta_+^{\op{c}}\left(e^{ - \pi \sqrt{-1} c_1(\pi^*\cE)/z} z^{-Gr} \hat \Gamma_\cY  (2 \pi \sqrt{-1})^{\deg_0/2} I^*(\ch^{\op{c}}(F)) \right)\\ \nonumber = &
 z^{-Gr} \hat \Gamma_\cZ (2 \pi \sqrt{-1})^{\deg_0/2} I^*(\ch(j^* \circ \pi_*(F))).
\end{align}
Finally, note that $\dim(\cY) = \dim(\cZ) + 2\op{rk}(\cE)$.
The claim follows from this, \eqref{e:comc}, and \eqref{e:gamma2}.
\end{proof}
\begin{rema}[Relation to \cite{IMM}]\label{r:IMMr}
A very similar statement was shown in \cite[Theorem 3.13]{IMM} and the proof above uses the same ingredients.  Indeed the statement above may be seen implicitly in the results of \cite{IMM} as we explain below.
Among other things, they show that the ambient quantum connection $\nabla^\cZ$ is related to $\nabla^{e^{-1}(\cE^\vee)}$ by the functor $j_*$ after a twist by $\op{det}(\cE)[\op{rk}(\cE)]$.  After composing with the pullback $\pi^*$, that result is essentially the adjoint to Theorem~\ref{t:dcom}, the observation of which almost gives a second proof of Theorem~\ref{t:dcom}, via Proposition~\ref{p:cvsreg} and the relations in Section~\ref{p:cvsreg}.

Note, however, that the statements differ further in the change of variables. One is the inverse of the other which, to the author's knowledge, is most easily seen  \emph{a-posteriori} by comparing the statements of  \cite[Theorem 3.13]{IMM} and Theorem~\ref{t:dcom} above.
The presentation given above implies more directly the results below involving the narrow quantum $D$-module of $\cY$, and the change of variables given above is designed to be useful in applications such as the LG/CY correspondence of \cite{Sh}.
\end{rema}

\subsection{Narrow quantum Serre duality}\label{s:narqsd}
In this section we prove a variation of quantum Serre duality which gives an isomorphism between the narrow quantum $D$-module of $\cY$ and the ambient quantum $D$-module of $\cZ$.
An application of this  theorem is given in \cite{Sh}.

\begin{defi}
Define the map $\hat f^\cY : H^*_{\op{CR}}(\cY) \to H^*_{\op{CR}, \op{nar}}(\cY) \otimes P^\cY$ by
\[ \hat f^\cY(\bt) := \sum_{d \in \op{Eff}} \sum_{k \geq 0} \frac{1}{k!} I_* \circ{ev_2}_*\left( ev_1^*(e(\cE^\vee)) \cup \prod_{j=3}^{k+2} ev_j^*(\bt) \cap \left[\sMbar_{0,k+2}(\cc Y, d)\right]^{vir} \right).\]  
Since the evaluation maps are proper and $e(\cE^\vee) \in H^*_{\op{CR}, \op{nar}}(\cY)$, $\hat f^\cY(\bt)$ will also lie in the narrow cohomology.
\end{defi}
\begin{lemma}\label{l:fY}
In $H^*_{\op{CR}}(\cY) \otimes P^{\cY}$, \[\hat f^\cY(\bt) = i_* (\hat f^\cX(i^* (\bt)).\]
\end{lemma}
\begin{proof}
Using~\eqref{e:equivinv} and~\eqref{e:equivpair} and the same analysis as in the proof of Proposition~\ref{p:noneq2} , $$\hat f^\cY(\bt) = \lim_{\lambda \mapsto 0} \pi^*( \hat f^\lambda(i^*(\bt))).$$  Next note that \begin{align*} i_* ( \hat f^\lambda(\bt)/e_\lambda(\cE^\vee)) &= \pi^*(i^*(i_* ( \hat f^\lambda(\bt)/e_\lambda(\cE^\vee)))) \\
&= \pi^*( \hat f^\lambda(\bt))
\end{align*} and $\hat f^\cX(\bt) = \lim_{\lambda \mapsto 0} \hat f^\lambda(\bt)/e_\lambda(\cE^\vee)$.
\end{proof}

%

\begin{defi}. 
We define the transformation $\Delta_+: H^*_{\op{CR}, \op{nar}}(\cY) \to H^*_{\op{CR}, \op{amb}}(\cZ)$ as follows.  Given $\alpha \in H^*_{\op{CR}, \op{nar}}(\cY)$, let $\tilde \alpha \in H^*_{\op{CR}, \op{c}}(\cY)$ be a lift of $\alpha$.  Define
\[\Delta_+(\alpha) :=\Delta_+^{\op{c}}(\tilde \alpha).\]
Define
\[\bar \Delta_+ := (2 \pi \sqrt{-1} z)^{\op{rk}(\cE)} \Delta_+.\]
\end{defi}

\begin{lemma} With assumption~\ref{a:ass2},
the map $\Delta_+: H^*_{\op{CR}, \op{nar}}(\cY) \to H^*_{\op{CR}, \op{amb}}(\cZ)$ described above is well-defined.
\end{lemma}
\begin{proof}
The lift $\tilde \alpha$ is only defined up to an element of $\ker(\phi)$.  We check that
\[\pi^{\op{c}}_*(\ker(\phi)) \subset \ker(j^*).\]  

Given $\tilde \alpha \in \ker(\phi)$, since $i^{\op{c}}_*$ is an isomorphism there exists an element $\beta \in H^*_{\op{CR}}(\cX)$ such that $\tilde \alpha = i^{\op{c}}_*(\beta)$.  Then $\pi^{\op{c}}_*(\alpha) = \pi^{\op{c}}_*(i^{\op{c}}_*(\beta)) = \beta$.  We want to show that $j^*(\beta) = 0$.  By assumption, $\phi (\tilde \alpha) = \phi \circ i^{\op{c}}_*(\beta) = i_*(\beta) = 0$.  Write $\beta$ as $i^*(\gamma)$, for some $\gamma \in H^*_{\op{CR}}(\cY)$.  Consider the following diagram
\[
\begin{tikzcd}
\cY|_\cZ \ar[r, "\tilde j"] \ar[d, "\tilde \pi"] & \cY \ar[d, "\pi"] \\
\cZ \ar[r, "j"] \ar[u, bend left, "\tilde i"]& \cX.  \ar[u, bend left, "i"]
\end{tikzcd}
\]
Then $j^*(\beta) = j^* i^*(\gamma) = \tilde i^* \tilde j^*(\gamma)$ is zero if and only if $\tilde j^*(\gamma) = 0$.  By Assumption~\ref{a:ass2}, $\tilde j^* (\gamma) = 0$ if and only if $\tilde j_* \tilde j^* (\gamma) = e(\cE) \cup \gamma = 0$.   Up to a sign, this is equal to $e(\cE^\vee) \cup \gamma = i_* i^* (\gamma) =i_*(\beta)$, which is zero by assumption.
\end{proof}
\begin{lemma}\label{l:deltadiff}
Given $\alpha = e(\cE^\vee) \cup \beta \in H^*_{\op{CR}}(\cX)$, 
\[\Delta_+(\pi^* (\alpha)) = j^* \beta.\] In particular, $\Delta_+: H^*_{\op{CR}, \op{nar}}(\cY) \to H^*_{\op{CR}, \op{amb}}(\cZ)$ is an isomorphism. 
\end{lemma}
\begin{proof}
Observe that 
\begin{align*} \pi^*(\alpha) &= e(\cE^\vee) \cup \pi^*(\beta) \\
&= i_* \circ i^* \circ \pi^*(\beta) \\
& =i_*(\beta) \\
& = \phi \circ i^{\op{c}}_*(\beta).\end{align*}
Therefore, 
$\Delta_+(\pi^* (\alpha)) = j^* \circ \pi^{\op{c}}_* \circ i^{\op{c}}_*(\beta) = j^*(\beta)$.  The second claim follows from 
\begin{align*} H^*_{\op{CR}, \op{amb}}(\cZ)  &= \op{im}(j^*) \\
&\cong H^*_{\op{CR}}(\cX)/(\ker (-\cup e(\cE^\vee)) \\
&\cong \op{im}(-\cup e(\cE^\vee)) \\
&= \pi^*(\op{im}(-\cup e(\cE^\vee))) \\
&= H^*_{\op{CR}, \op{nar}}(\cY).
\end{align*} where the second and third terms are isomorphic by \eqref{e:ass2} and the final equality is by Proposition~\ref{p:tot}.
\end{proof}

%
%
%
We will need the following lemma.
\begin{lemma}\label{l:noncov}
$\Delta_+(\hat f^\cY(\bt)) = j^*\circ \hat f ^\cX\circ i^*(\bt)$. 
\end{lemma}
\begin{proof}
By Lemma~\ref{l:fY} and the  definition of $\Delta_+$, 
\begin{align*}
\Delta_+ (\hat f^\cY(\bt)) & = \Delta_+(i_*(\hat f^\cX(i^*(\bt))) \\
 &= \Delta_+ ( \phi \circ i^{\op{c}}_*(\hat f ^\cX(i^*(\bt)) ))\\
& = 
j^* \circ \pi^{\op{c}}_* (i^{\op{c}}_* ( \hat f ^\cX(i^*(\bt)))) \\
& = j^*\circ \hat f ^\cX\circ i^*(\bt).\end{align*}
\end{proof}

\begin{prop}\label{p:Lcom+}
The following operators are equal after a change of variables: 
\begin{align*}  \Delta_+ \circ L^{\cY}(\bt, z) \circ \phi &=  
\Delta_+^{\op{c}}  \circ L^{\cY, \op{c}}(\bt, z)  \\
& = L^{\cZ, \op{amb}}(j^* \circ \bar f^\cX\circ i^*(\bt), z) \circ \Delta^{\op{c}}_+ \circ e^{- \pi \sqrt{-1} c_1(\pi^*\cE)/z} \\
&= L^{\cZ, \op{amb}}(\bar  f^\cY(\bt), z)  \circ \Delta_+\circ e^{- \pi \sqrt{-1} c_1(\pi^*\cE)/z} \circ \phi,
\end{align*}
where 
\begin{align}\label{e:barf} \bar f^\cY (\bt) &= 
\nonumber  \Delta_+ ( \hat f^\cY(\bt)) - \pi \sqrt{-1} c_1(\cE) \\
 &=\Delta_+ \left( \sum_{i \in I_{\op{nar}}} \langle \langle e(\cE^\vee), T_i \rangle \rangle^{\cY}(\bt ) T^i \right) - \pi \sqrt{-1} c_1(\cE).
\end{align}
\end{prop}
\begin{proof}
This follows almost immediately from the previous section, by applying the map $\phi: H^*_{\op{CR}, \op{c}}(\cY) \to H^*_{\op{CR}, \op{nar}}(\cY)$.  Recall \eqref{e:comc}:
\begin{equation}\label{e:rec}
\Delta^{\op{c}}_+ \circ L^{\cY, \op{c}}(\bt, z) (\beta) = L^{\cZ, \op{amb}}(j^* \circ \bar f^\cX\circ i^*(\bt), z)\circ \Delta^{\op{c}}_+ \circ e^{- \pi \sqrt{-1} c_1(\pi^*\cE)/z} (\beta).
\end{equation}
By Proposition~\ref{p:phi}, the left hand side of \eqref{e:rec} is equal to
\begin{equation*}
\Delta_+^{\op{c}} \circ \phi \left(L^{\cY, \op{c}}(\bt, z) (\beta) \right) = \Delta_+ \circ L^{\cY}(\bt, z) (\phi(\beta) )
\end{equation*}
for all $\beta \in H^*_{\op{CR}, \op{c}}(\cY)$.  By Lemma~\ref{l:noncov} the right hand side of \eqref{e:rec} is 
\begin{align*}
&L^{\cZ, \op{amb}}(\bar f^\cY(\bt), z)\circ  \Delta^{\op{c}}_+ \circ e^{- \pi \sqrt{-1} c_1(\pi^*\cE)/z} (\beta)
\\ = &L^{\cZ, \op{amb}}(\bar f^\cY(\bt), z) \circ\Delta_+ \circ e^{- \pi \sqrt{-1} c_1(\pi^*\cE)/z} (\phi( \beta)).
\end{align*}
\end{proof}

\begin{theo}\label{t:QSD}  Assume $\cX$ has the resolution property.
The map $\bar \Delta_+$ identifies the quantum $D$-module $QDM_{\op{nar}}(\cY)$ with $\bar f^* \left( QDM_{\op{amb}}(\cZ) \right)$.  Furthermore it is compatible with the integral structure and the functor $j^* \circ \pi_*$, i.e., the following diagram commutes;
\begin{equation}\label{e:comn}
\begin{tikzcd}
D(\cY)_{\cX} \ar[r, " j^* \circ \pi_*"] \ar[d, "s^{\cY, \op{nar}}"] &  j^*(D(\cX))  \ar[d, "s^{\cZ, \op{amb}}"]\\
QDM_{\op{nar}}(\cY) \ar[r, "\bar \Delta_+"] & QDM_{\op{amb}}(\cZ).
\end{tikzcd}
\end{equation}
%
\end{theo}

\begin{proof}  The fact that $\nabla^{\cY, \op{nar}}$ is mapped to $\nabla^{\cZ, \op{amb}}$ follows from Proposition~\ref{p:Lcom+}.
To see that the pairings agree, first observe that for $\alpha, \beta \in H^*_{\op{CR}}(\cX)$,
\begin{align*}
 \langle \Delta_+ i_* \alpha, \Delta_+ i_* \beta \rangle^{\cZ}  &= \int_{I\cZ} j^* (\alpha) \cup_{I\cZ} I_*(j^*  \beta) \\
 & =  \int_{I\cX} \alpha \cup_{I\cX} I_*(\beta \cup_{I\cX} q^*e(\cE)) \\
 & = (-1)^{\op{rk}(\cE)} \int_{I\cX} \alpha \cup_{I\cX}  I_*\left({i^{\op{c}}}^* i^{\op{c}}_*(\beta)\right) \\  & = (-1)^{\op{rk}(\cE)} \int_{I\cY} i^{\op{c}}_* \alpha \cup_{I\cY} I_*\left( i^{\op{c}}_* \beta\right)  \\
  & = (-1)^{\op{rk}(\cE)} \int_{I\cY} i_* \alpha \cup_{I\cY, \op{c}} I_*\left( i_* \beta \right)  \\
 &= (-1)^{\op{rk}(\cE)} 
 \langle i_*\alpha, i_* \beta \rangle^{\cY},
 \end{align*} 
 where $\cup_{I\cZ}$ (resp. $\cup_{I\cY}$, etc...) indicates we are using the usual cup product on $I\cZ$ (resp. $I\cY$, etc...) as opposed to the Chen--Ruan cup product.
The fourth equality is the projection formula (Proposition 6.15 of \cite{BT}). In the fifth equality we use the fact that   
 $i_* \alpha \cup_{I\cY, \op{c}} I_*\left( i_* \beta \right) = (\phi \circ i^{\op{c}}_* \alpha) \cup_{I\cY, \op{c}} I_*\left( \phi \circ  i^{\op{c}}_* \beta\right)= i^{\op{c}}_* \alpha \cup_{I\cY} I_*\left( i^{\op{c}}_* \beta\right)$ by Definition~\ref{d:cupcom} of the compactly supported cup product.  Because $\bar \Delta_+$ contains the factor of $z^{\op{rk}(\cE)}$,
 \begin{align*}
 S^\cZ(\bar \Delta_+ i_* \alpha, \bar \Delta_+ i_* \beta) & = (-1)^{\op{rk}(\cE)}(2 \pi \sqrt{-1} z)^{\op{dim}(\cZ) + 2\op{rk}(\cE)} \langle \Delta_+ i_* \alpha, \Delta_+ i_* \beta \rangle^{\cZ}\\ &= (2 \pi \sqrt{-1} z)^{\op{dim}(\cY)} \langle i_*\alpha, i_* \beta \rangle^{\cY}\\
 &= S^\cY (i_*\alpha, i_* \beta).
 \end{align*}
 Because $H^*_{\op{CR}, \op{nar}}(\cY) = \op{im}(i_*)$ by Proposition~\ref{p:tot}, we conclude that $\bar \Delta_+$ preserves the pairing on all of $H^*_{\op{CR}, \op{nar}}(\cY)$.
 
 Commutativity of the square \eqref{e:comn} follows from the equalities of Proposition~\ref{p:Lcom+} and Theorem~\ref{t:dcom} after observing that $\ch = \phi \circ \ch^{\op{c}}$.
\end{proof}

\section{Appendix 1: Givental's formalism}

Let $\square$ denote either Gromov--Witten theory of a proper Deligne--Mumford stack $\cX$
or an $\bs$-twisted theory over $\cX$.  In this section we recall Givental's formalism of an overruled Lagrangian cone for encoding genus-zero Gromov--Witten theory.  For more details see \cite{G3}.

Let $H^\square$ be the state space for the theory. 
Given a basis $\{T_i\}_{i \in I}$, a point in $H^\square$ may be written as $\bt = \sum_{i \in I} t^i T_i$.
Denote by $\bt(z)$ the power series 
\[\bt(z) = \sum_{k \geq 0} \bt_k z^k = \sum_{i \in I} \sum_{k \geq 0} t^i_k T_i z^k\]
in $H^\square[[z]]$.  Define the genus zero descendent potential as the formal function
\[ \cc F^\square_0(\bt(z)) := \sum_{d \in \op{Eff}} \sum_{n \geq 0} \frac{1}{n!} \langle \bt(\psi), \ldots, \bt(\psi)\rangle_{0, n,d}^\square.\]

Let $\sV^\square$ denote the infinite-dimensional vector space $H^\square[z][[z^{-1}]][[\bs]]$ (where $\bs = 0$ in the untwisted case).  We endow $\sV^\square$ with a symplectic pairing defined as follows:
\[\Omega^\square(f_1(z), f_2(z)) := \op{Res} \langle f_1(-z), f_2(z)\rangle^\square.\]
The vector space $\sV^\square$ may be polarized as 
\begin{align*} \sV^\square_+ &= H^\square[z],\\
\sV^\square_- & = z^{-1} H^\square[[z^{-1}]].
\end{align*}
The polarization on $\sV^\square$ determines Darboux coordinates $\set{q_k^i, p_{k,i}}$. Each
element of $\sV^\square$ may be written as 
\[
\sum_{k \geq 0}\sum_{i \in I} q_k^i T_i z^k + \sum_{k \geq 0}\sum_{i \in I} p_{k,i}T^i (-z)^{-k-1}
\]
We view $\cc F^\square_0(\bt(z))$ as a formal function on $\sV^\square_+$ via the \emph{dilaton shift}
\[\bq(z) = \bt(z) - z.\]  
\begin{defi}
Define the \emph{overruled Lagrangian cone} for $\square$ to be
\begin{equation}\label{e:cone1}
\sL^\square :=\{\bp =d_{\bq} \cc F^\square_0\}.
\end{equation}
Explicitly, $\sL^\square$ contains the points of the form
\begin{equation}\label{e:cone2}
- z+ \sum_{\substack{k \geq 0 \\ i \in I}} t_k^i T_i z^k +\sum_{d \in \op{Eff}} \sum_{\substack{a_1, \ldots , a_n,  a\geq 0 \\ i_1, \ldots , i_n, i \in I}} \frac{t^{i_1}_{a_1}\cdots t^{i_n}_{a_n}}{n!(-z)^{a+1}}\langle \psi^aT_i ,\psi^{a_1}T_{i_1},\dots,\psi^{a_n}T_{i_n}\rangle_{0, n+1,d}^\square T^i.
\end{equation}
\end{defi}
As shown in \cite{G3}, $\sL^\square$ takes a special form: 
\begin{itemize}
\item it
is a cone; 
\item for all $f \in \sL^\square$,
\[\sL^\square \cap T_f\sL = zT_f\sL\] 
where $T_f\sL$ is the tangent space to $\sL^\square$ at $f$.
\end{itemize} 

Consider a generic family in $\sL^\square$ parameterized by $H^\square$, this will take the form 
\[
\{f(\bt)| \bt \in H^\square\} \subset \sL^\square,
\]
and
will be transverse to the ruling. 
With this, the above properties imply that we can reconstruct $\sL^\square$ as
\begin{equation}\label{e:genfam}
\sL^\square = \set{zT_{f(\bt)}\sL^\square| \bt \in H^\square}.
\end{equation}

Givental's $J$--function is such a family.  It is given by the intersection:
\[
J^\square(\bt, -z) = \sL^\square \cap -z \oplus \bt \oplus \sV^{-}.
\]
More explicitly, 
\[
J^\square(\bt, -z) = -z + \bt +\sum_{d \in \op{Eff}} \sum_{n \geq 0} \sum_{ i \in I} \frac{1}{n!} \br{ \frac{T_i}{-z - \psi}, \bt, \ldots, \bt }^\square_{0, n+1,d} T^i.
\]  
By \eqref{e:invadj}, we see that
\begin{equation}\label{e:LJ}
L^\square(\bt, z)^{-1} \alpha = L^\square(\bt, -z)^T \alpha = \partial_\alpha J^\square(\bt, z).
\end{equation}

In \cite{G3} it is shown that the image of $J^\square(\bt, -z)$ is transverse to the ruling of $\sL^\square$, so $J^\square(\bt, -z)$ is a function satisfying \eqref{e:genfam}.  
It follows that the ruling at $J^\square(\bt, -z)$ is spanned by the derivatives of $J^\square$, i.e.

\begin{align}\label{e:Jgens}zT_{J^\square(\bt, -z)}\sL &= \set{J^\square(\bt, -z) + z\sum c_i(z) \frac{\partial}{\partial t^i} J^\square(\bt, -z) | c_i(z) \in \CC[z]} \\
&= \set{z\sum c_i(z) \frac{\partial}{\partial t^i} J^\square(\bt, -z) | c_i(z) \in \CC[z]}. \nonumber
\end{align}
where the second equality is by the string equation, $ z\frac{\partial}{\partial t^0} J^\square(\bt, z) = J^\square(\bt, z)$.  We note finally that
\begin{equation}\label{e:Jgens2}
T_{J^\square(\bt, -z)}\sL = \set{ \sum c_i(z) \frac{\partial}{\partial t^i} J^\square(\bt, -z) | c_i(z) \in \CC[z]}.
\end{equation}

\subsection{Quantum Serre duality with Lagrangian cones}
Since its discovery in \cite{G1}, quantum Serre duality (or non-linear Serre duality) has been formulated in many different ways.  Below we recall one of the most general and applicable, in terms of twisted theories and Lagrangian cones.  
Let $\bs$ denote the twisting parameters of \S~\ref{s:td} and define $\bs^*$ by 
\[s_k^* = (-1)^{k+1} s_k.\]
Note \[ \bs^*(\cE^\vee) = \frac{1}{\bs(\cE)}.\]
In this case (genus zero) quantum Serre duality takes the following form.
\begin{theo}\cite[Corollary 9]{CG}
The symplectic transformation
\begin{align*} \sV^{\bs(\cE)} &\to \sV^{\bs^*(\cE^\vee)}\\
 f(z) &\mapsto \bs^*(\cE^\vee) f(z)\end{align*}
identifies $\sL^{\bs(\cE)}$ with $\sL^{\bs^*(\cE^\vee)}$.
\end{theo}
See  \cite{Ts} for the orbifold version of this theorem.

Note however that the specializations of the twisting parameters given by \eqref{e:specCI} and \eqref{e:spec2} are not exactly of the form given above, i.e. $\bs ' \neq \bs^*$.  The statement must be modified slightly in this case.  
We state here the modified statement of the above theorem for the case of the Euler class-twisted theories of \S~\ref{s:twisting}.   
This specific formulation appears as Theorem 5.17 of~\cite{LPS}.

\begin{theo}\cite[Theorem 5.17]{LPS} \label{t:QSDlag} The symplectic transformation
\begin{align*}
\Delta^\diamond: \sV^{e_\lambda^{-1}(\cE^\vee)} &\to \sV^{e_\lambda(\cE)}  \\
f(z) &\mapsto \frac{e^{ \pi \sqrt{-1} c_1(\cE)/z}}{e_{\lambda}(\cE^\vee)}f(z)
\end{align*}
identifies $\sL^{e_\lambda^{-1}(\cE^\vee)}$ with $\sL^{e_\lambda(\cE)}$.
\end{theo}

\section{Appendix 2: Orbifold localized Chern character}

Given a complex $F^\bullet \in K^0_X(Y)$ (exact off $X$), there exists a \emph{localized Chern character}
\[ \chr^Y_X(F^\bullet) \in H^*(Y, Y- X)\]
as described in Example 19.2.6 of \cite{Fu} (see also \cite{Iv, ABS}).

On the other hand, given a bundle $F$ on a stack $\cY$ with the resolution property, there is an \emph{orbifold Chern character} \cite{To, Ts} landing in the cohomology of the inertia stack and defined as follows.  Restricting $F$ to a twisted sector $F_\gamma \to \cY_\gamma$, $F_\gamma$ decomposes into eigenbundles according to the action of $\gamma$ on $F$
\[F_\gamma = \oplus_{0 \leq f < 1} F_{\gamma, f},\]
where the generator $\gamma$ acts as multiplication by $e^{2 \pi \sqrt{-1} f}$ on $F_{\gamma, f}$.  Define 
\[ \rho(F_\gamma) := \sum_{0 \leq f < 1} e^{2 \pi \sqrt{-1} f} F_{\gamma, f}.\]
Observe that for a complex 
\[F^\bullet =  0 \to F^a \to \cdots  \to F^i \to F^{i+1} \to \cdots \to F^b \to 0\]
of vector bundles, the map $d^i_\gamma: F^i_\gamma \to F^{i+1}_\gamma$ is compatible with the splitting into eigenbundles.  I.e. $d^i_\gamma$ is a direct sum of the maps 
\[d^i_{\gamma, f}: F^i_{\gamma, f} \to F^{i+1}_{\gamma, f}.\]
Consequently $\rho$ gives a well-defined map on $K^0(\cY_\gamma)$.  

Summing over each twisted sector, this defines a map $\rho: K^0(I\cY) \to K^0(I\cY)$.
\begin{defi} The orbifold Chern character $\ch: K^0(\cY) \to H^*_{\op{CR}}(\cY) = H^*(I\cY)$ is defined as the composition
\[K^0(\cY) \xrightarrow{q^*}K^0(I\cY) \xrightarrow{\rho} K^0(I\cY) \xrightarrow{\chr} H^*(I\cY),\]
where $q: I\cY \to \cY$ is the natural union of inclusions and $\chr$ is the usual Chern character defined by passing to the coarse space.
\end{defi}
One can combine the two notions above to obtain a \emph{localized orbifold Chern character}.  Let $\cX$ be a closed substack of $\cY$ and let $F^\bullet$ be a complex on $\cY$, exact off of $\cX$.  Consider the restriction to a twisted sector $\cY_\gamma$, by the observation above, $F^\bullet_\gamma$ splits into eigen-complexes
\[F^\bullet_\gamma = \oplus_{0 \leq f <1} F^\bullet_{\gamma, f},\]
with each $F^\bullet_{\gamma, f}$ exact off $\cX_\gamma$.  This implies that the twisting $\rho$ gives a well defined map on $K^0_{\cX_\gamma}(\cY_\gamma)$.  Summing over all twisted sectors defines a map $\rho: K^0_{I\cX}(I\cY) \to K^0_{I\cX}(I\cY)$.
\begin{defi}
Define the \emph{localized orbifold Chern character} $$\ch^\cY_\cX: K^0_{\cX}(\cY) \to H^*_{\op{CR}}(\cY, \cY - \cX)$$ to be the composition 
\[K^0_\cX(\cY) \xrightarrow{q^*}K^0_{I\cX}(I\cY) \xrightarrow{\rho} K^0_{I\cX}(I\cY) \xrightarrow{\chr^\cY_\cX} H^*(I\cY, I\cY - I\cX).\]
\end{defi}

For $Y$ a non-compact manifold, define $K^0_{\op{c}}(Y)$ to be the direct limit
\[ K^0_{\op{c}}(Y) = \lim_{\to} K^0_X(Y)\]
over all compact subvarieties $X \subset Y$.  Assume $X_1 \subset X_2 \subset Y$ and $F^\bullet$ is a complex exact off $X_1$ (and therefore also exact off $X_2$), let $j: X_1 \to X_2$, $i_1: X_1 \to Y$, and $i_2: X_2 \to Y$ denote the inclusions.
Then
the following diagram commutes:
\begin{equation}\label{e:Kcom}
\begin{tikzcd}
K^0_{X_1}(Y) \ar[r, "j_*"] \ar[d, "\chr^Y_{X_1}"] & K^0_{X_2}(Y) \ar[d, "\chr^Y_{X_2}"] \\
H^*(Y, Y- X_1)  \ar[d, " \cap {\left[Y\right]}"] &  H^*(Y, Y- X_2)\ar[d, " \cap {\left[Y\right]}"] \\
H_*(X_1) \ar[r, "j_*"] \ar[d, "{i_1}^{\op{c}}_*"] & H_*(X_2)\ar[d, "{i_2}^{\op{c}}_*"] \\
H_*(Y) \cong H^*_{\op{c}}(Y) \ar[r] &H_*(Y) \cong H^*_{\op{c}}(Y).
\end{tikzcd}
\end{equation}
The commutativity of the top square follows, for instance, from Definition 18.1 and Example 19.2.6 of \cite{Fu}.
\begin{defi} For $i: X \to Y$ the inclusion of a closed and proper subvariety,
define $\chr^{\op{c}}_X: K^0_X(Y) \to H^*_{\op{c}}(Y)$ by
\[\chr^{\op{c}}_X(F^\bullet) = i^{\op{c}}_* \left( \chr^Y_X(F^\bullet) \cap [Y]\right).\] By diagram \eqref{e:Kcom}, this induces a homomorphism
\[\chr^{\op{c}}: K^0_{\op{c}}(Y) \to H^*_{\op{c}}(Y)\]
which we will refer to as the \emph{compactly supported Chern character}.  
\end{defi}

The above argument  can be extended to the situation where $\cY$ is a smooth Deligne--Mumford stack, we obtain:
\begin{defi}\label{d:cch}
The 
\emph{compactly supported orbifold Chern character} \[\ch^{\op{c}}: K^0_{\op{c}}(\cY) \to H^*_{\op{CR}, \op{c}}(\cY) = H^*_{\op{c}}(I\cY)\] is defined to be the composition
\[K^0_{\op{c}}(\cY) \xrightarrow{q^*}K^0_{\op{c}}(I\cY) \xrightarrow{\rho} K^0_{\op{c}}(I\cY) \xrightarrow{\chr^{\op{c}}} H^*_{\op{c}}(I\cY).\]
\end{defi}
Given $i: X \to Y$ the inclusion of a closed and proper subvariety, by Theorem 1.3 of \cite{Iv},
\[ i_*\left( \chr^Y_X(F^\bullet\right) \cap [Y] = \chr(F^\bullet).\]
This immediately gives the following.
\begin{prop}\label{p:cnar}
For $F^\bullet$ a complex of vector bundles on $Y$ with proper support,
\[ \phi ( \ch^{\op{c}}(F^\bullet)) = \ch(F^\bullet)\]
in $H^*_{\op{CR}}(\cY)$.
\end{prop}

\bibliographystyle{plain}
\bibliography{references}
\end{document}